\documentclass[a4paper,9pt,oneside]{article}
\usepackage[utf8]{inputenc}
\usepackage[english]{babel}
\usepackage[utf8]{inputenc}
\usepackage[T1]{fontenc}
\usepackage{amsmath}
\usepackage{amssymb}
\usepackage{amsthm}
\usepackage{mathtools}
\usepackage{graphicx}
\usepackage[shortlabels]{enumitem}
\usepackage{tikz}
\usepackage{mathabx}

\theoremstyle{plain}
\newtheorem{thm}{Theorem}
\newtheorem{prop}[thm]{Proposition}
\newtheorem{lemma}[thm]{Lemma}
\newtheorem{cor}[thm]{Corollary}

\theoremstyle{definition}
\newtheorem{defin}[thm]{Definition}
\newtheorem{ex}[thm]{Example}
\newtheorem{rem}[thm]{Remark}

\newcommand{\X}{\operatorname{\mathfrak{X}}}
\newcommand{\N}{\operatorname{\mathbb{N}}}
\newcommand{\D}{\operatorname{\mathbb{D}}}
\newcommand{\Z}{\operatorname{\mathbb{Z}}}
\newcommand{\R}{\operatorname{\mathbb{R}}}
\newcommand{\C}{\operatorname{\mathbb{C}}}
\newcommand{\U}{\operatorname{\mathbb{U}}}
\newcommand{\hi}[1]{\lgroup{#1}\rgroup}
\newcommand{\Ho}[1]{{H}_{#1}}
\newcommand{\HBo}[1]{{H\!B}_{#1}}
\newcommand{\B}[1]{\mathcal{B}_{\alpha}^{#1}}
\newcommand{\supp}{\mathrm{supp}\,}
\newcommand{\Lda}{L^2_\alpha}
\newcommand{\de}{{\rm d}}

\mathtoolsset{showonlyrefs}

\title{Horocyclic harmonic Bergman spaces on homogeneous trees}

\author{Filippo De Mari\thanks{Dipartimento di Matematica, Dipartimento di Eccellenza 2023-2027, and MaLGa center, Università di Genova, Via Dodecaneso 35, Genova, Italy, email: {demari@dima.unige.it}, {rizzo@dima.unige.it}} \and Matteo Monti\thanks{Dipartimento di Scienze Matematiche “Giuseppe Luigi Lagrange”, Politecnico di Torino, Corso Duca degli Abruzzi 24, 10129, Torino, Italy, email: \mbox{matteo.monti@polito.it}} \and Elena Rizzo\footnotemark[1]}

\begin{document}

\maketitle

\begin{abstract}
The main focus of this contribution is on the harmonic Bergman spaces $\B{p}$ on the $q$-homogeneous tree $\X_q$ endowed with a family of measures $\sigma_\alpha$  that  are constant on the horocycles tangent to a fixed boundary point and turn out to be doubling with respect to the corresponding horocyclic Gromov distance. A central role is played by the reproducing kernel Hilbert space $\B{2}$ for which we find a natural orthonormal basis and  formulae for the kernel. We also consider the atomic Hardy space and the bounded mean oscillation space. Appealing to an adaptation of Calder\'on-Zygmund theory and to standard boundedness results for integral operators on $L^p_\alpha$ spaces with H\"ormander-type kernels, we determine the boundedness properties of the Bergman projection. This work was inspired by \cite{ccps, dmmv}.
\end{abstract}

\section*{Introduction}
\let\thefootnote\relax\footnotetext{\textbf{Keywords}: Bergman spaces, homogeneous trees, horocycles, Bergman projectors, Integral operators.}
\let\thefootnote\relax\footnotetext{\textbf{Mathematics Subject Classification} (2010): 05C05; 46E22; 43A85.}
	
In the literature, homogeneous trees are often studied as discrete counterparts of symmetric spaces of rank one. Among other reasons, this is also motivated by the existence of a natural embedding of the $q$-homogeneous tree $\X_q$ into the hyperbolic disk $\D$,~\cite{cc}. In the last years, this analogy has been exploited in both directions in order to spread concepts and problems from one setting to the other and vice versa. In this perspective, holomorphic Bergman spaces and projections are extensively studied in the continuous setting (see, for example,~\cite{fr,s2,z}) and variants thereof have been introduced on homogeneous trees only in recent developments. Since a definition of holomorphic function on discrete structures is not clearly stated in the literature, in this setting the holomorphicity has been replaced by harmonicity.
In~\cite{ccps,ccps2} the authors define the harmonic Bergman spaces on $\X_q$ as the families of harmonic functions that lie in some $L^p$ space for a so-called reference measure. A reference measure is a finite positive measure which is radial with respect to a fixed vertex and decreasing with respect to the distance from it. The radiality of such measures mimics the classical Bergman measures considered on the hyperbolic disk, that are given by $(1-|w|^2)^{\alpha-2}\de w$, $\alpha>1$.

It is well known that $\D$ can be realized in a different way as the upper half plane $\U$, where the classical Bergman measures are ${\rm Im} (z)^{\alpha-2}\de z$, $\alpha>1$. It is clear that in this version the measures are no longer radial, while they are constant on the horizontal lines, namely the horocycles tangent to the point at infinity. Therefore, in analogy with \cite{ccps, dmmv}, we define a family of measures on $\X_q$ by
\begin{equation}\label{misintro}
    \sigma_\alpha(x)=q^{-\alpha\hi{x}_\omega}, \quad \alpha>1,
\end{equation}
where $\hi{x}_\omega\in\Z$ is the horocyclic index of the vertex $x$ with respect to the fixed boundary point $\omega$. All measures in this family are obviously constant on horocycles tangent to $\omega$ and should be interpreted as a discrete counterpart of the horocyclic Bergman measures on $\U$. Throughout the paper we do not change the geometry of $\X_q$ but we adopt a different perspective with respect to \cite{ccps, dmmv}, by considering a boundary point instead of a vertex as reference point. Indeed, in \cite{ccps, dmmv} the fixed vertex plays the same role as the origin in $\D$, while here the boundary point $\omega$ will take the place of the point at infinity of $\U$.

On $\X_q$, as well as on $\D$ and $\U$, the Bergman space $\mathcal B_\alpha^2$ is a reproducing kernel Hilbert space, so that there exist a kernel $\mathcal K_\alpha$ and a Bergman projection $P_\alpha\colon L^2_\alpha\to\mathcal B_\alpha^2$ given by integration against $\mathcal K_\alpha$. The boundedness of $P_\alpha$ on the main function spaces of $\D$ and $\U$ is completely characterized, see~\cite{z} and~\cite{b}, respectively. In fact, it is well known that in these contexts, $P_\alpha$ extends to a bounded operator on $L^p_\alpha$ if and only if $1<p<\infty$. Furthermore, $P_\alpha$ is of weak type $(1,1)$ and extends to a bounded operator from the atomic Hardy space $H^1_\alpha$ to $L^1_\alpha$ and from $L^\infty_\alpha$ to the bounded mean oscillation space $BMO_\alpha$, for suitable definitions of $H^1_\alpha$ and $BMO_\alpha$. In \cite{dmmv} it is shown that analogous boundedness results hold true for a subclass of the measures introduced in \cite{ccps}. It is worth noticing that the techniques implemented in the continuous and in the discrete setting are radically distinct. The purpose of this paper is to obtain the same kind of boundedness properties for integral operators with respect to $\sigma_\alpha$. Observe that our measures are substantially different from those studied in \cite{dmmv} not just because they are constant on horocycles instead of being radial, but rather because they are not finite.

The Bergman measures we introduce in this work do not seem to appear elsewhere in the literature. It is worth observing that formula~\eqref{misintro} for $\alpha=1$ provides a measure already considered in~\cite{atv,lstv}, which is the so-called canonical flow measure. In particular, in~\cite{lstv} the same theory is developed for the more general class of flow measures. All such measures are not doubling, but only locally doubling with respect to the usual graph distance. 
In the context of homogeneous trees, the boundedness of singular integrals associated with the combinatorial Laplacian and the counting measure has been investigated in~\cite{cms}, while various Hardy and $BMO$ spaces are studied in~\cite{atv,atv2,cm}.

A further development of this work could be to characterize the class of measures on homogeneous trees for which these results hold, in the spirit of the well-known papers \cite{b2,bb} for the hyperbolic disk. Moreover, inspired by \cite{hs,lstv}, among others, we would like to apply our work to a broader family of operators, such as suitable spectral multipliers and first order Riesz transforms associated to some Laplacian operators. Finally, another question is whether it is natural to produce an adapted version of this approach in the setting of non homogeneous trees.

The structure of the paper is the following. The first section is devoted to some preliminaries about homogeneous trees and to introduce harmonic functions on them. In Section 2 we start by defining the horocyclic Bergman measures $\sigma_\alpha$ and giving doubling properties of the resulting metric space with respect to two different metrics. Then we present the horocyclic harmonic Bergman spaces $\B{p}$, and we conclude with the definitions of the atomic Hardy space $H^1_\alpha$ and the bounded mean oscillation space $BMO_\alpha$. In Section 3 we provide an orthonormal basis for $\B{2}$. Inspired by this basis, in Section 4 we define a function on $\X_q$ that has a key role in finding explicit formulae for the reproducing kernel of $\B{2}$, that we state in Theorems \ref{rkv} and \ref{rkv2}. Finally, the last section is dedicated to present a Calder\'on-Zygmund decomposition and a suitable H\"ormander condition and to obtain boundedness of integral operators on the $L^p_\alpha$ spaces, as well as limiting results involving $H^1_\alpha$ and $BMO_\alpha$. Furthermore, we focus on the Bergman projection from $L^2_\alpha$ to $\B{2}$, showing that it satisfies all the hypotheses of these results.

Throughout the paper we will assume the convention that a sum vanishes whenever the starting index is larger than the ending one. The symbols $\lesssim$ and $\simeq$ indicate that the left hand side is smaller than and equal to, respectively, the right hand side multiplied by a uniform constant. Moreover, for every $t\in\R$, $\lfloor t \rfloor\in\Z$ denotes the floor function, namely the greatest integer less than or equal to $t$.

\section{Preliminaries}

In this section we will give some basic notions about homogeneous trees and harmonic functions on trees. For more details we refer to \cite{c,cms,cms2,ftn}.

\subsection{Homogeneous trees}
A pair $(\X,\mathfrak{E})$ is a \textit{graph} if $\X$ is a discrete set of elements, which are the \textit{vertices} of the graph, and $\mathfrak{E}$ is a set of paired vertices, called the \textit{edges} of the graph. If two vertices $x,y\in\X$ form an edge $\{x,y\}\in\mathfrak{E}$ we say that they are \textit{neighbours} and we write $x\sim y$. A \textit{path} is a finite sequence of vertices $x_0,\dots,x_n\in\X$ such that $x_i\sim x_{i+1}$ for every $i=0,\dots,n-1$, a \textit{chain} is a path where $x_i\ne x_{i+2}$ for every $i=0,\dots,n-2$ and a \textit{loop} is a chain with $x_0=x_n$. For every $q\in\N$, $q\ge2$, the $q$-homogeneous tree is a connected graph which has no loops and where every vertex has exactly $q+1$ neighbours. We will usually refer to the $q$-homogeneous tree just by the set of its vertices $\X_q$, and whenever $q$ is fixed we write $\X$ instead of $\X_q$. Notice that, for every $x,y\in\X_q$, there always exists a unique chain $[x,y]$ joining them, so that there is a natural distance on the tree, the \textit{edge counting distance} $d:\X_q\times\X_q\to\N$, such that $d(x,y)$ is the number of edges in the chain $[x,y]$.

Now fix $q\in\N$, $q\ge2$, denote by $\mathcal{I}(\X)$ the set of infinite chains in $\X$ and consider the following equivalence relation on $\mathcal{I}(\X)$. Given $(x_i)_{i\in\N}, (y_j)_{j\in\N}$ in $\mathcal{I}(\X)$ we say that they are equivalent if and only if there exists $m\in\Z$ such that $x_i=y_{i+m}$ eventually. The \textit{boundary} $\Omega$ of $\X$ is defined as the quotient space of $\mathcal{I}(\X)$ by this equivalence relation. The idea is that a boundary point can be seen as the limit point of an infinite chain in $\X$. From now on we fix a point $\omega\in\Omega$.

For $x,y\in\X$ we denote by $[x,\omega)$ and $[y,\omega)$ the infinite chains starting from $x$ and $y$, respectively, and belonging to the equivalence class of $\omega$. By the definition of $\Omega$, and in particular of the equivalence relation on $\mathcal{I}(\X)$, there always exists $z\in\X$ such that $[x,\omega)\cap[y,\omega)=[z,\omega)$. We call $z$ the \textit{confluent} of $x$ and $y$ and we denote it by $x\wedge_{\omega}y$. If $d(x,x\wedge_{\omega}y)=d(y,x\wedge_{\omega}y)$, we say that $x$ and $y$ belong to the same \textit{horocycle} with respect to $\omega$, which is an infinite subset of $\X$. The set of horocycles with respect to $\omega$ is indexed on $\Z$ and forms a partition of the tree. From now on, we only consider horocycles tangent to $\omega$. We fix a reference horocycle $\Ho{0}^{\omega}$, it is easy to see that for all $x\in\X$ the map $y\mapsto d(x,x\wedge_{\omega}y)-d(y,x\wedge_{\omega}y)$ is constant on $\Ho{0}^{\omega}$. Thus, the $k$-th horocycle $\Ho{k}^{\omega}$, $k\in\Z$, defined by
\begin{equation*}
    \Ho{k}^{\omega}:=\Bigl\{x\in\X: d(x,x\wedge_{\omega}y)=d(y,x\wedge_{\omega}y)+k \text{ for some }y\in\Ho{0}^{\omega}\Bigr\}
\end{equation*}
is well defined. Roughly speaking, we enumerate the horocycles with negative integers in the direction of $\omega$ and with positive integers in the opposite direction (see Figure~\ref{f1}). The sets of the form
\begin{equation*}
    H\!B_n^{\,\omega}:=\bigcup_{j\le n}\Ho{j}^{\omega}
\end{equation*}
are called \textit{horoballs}. For all $x\in\X$, its \textit{horocyclic index} $\hi{x}_{\omega}$ is the unique integer such that $x\in\Ho{\hi{x}_{\omega}}^{\omega}$, its \textit{predecessor} $p(x)\in\X$ is the only neighbour of $x$ which belongs to the horocycle $\Ho{\hi{x}_{\omega}-1}^{\omega}$ and the set $S(x)\subseteq\X$ of its \textit{successors} consists of the $q$ neighbours of $x$ belonging to $\Ho{\hi{x}_{\omega}+1}^{\omega}$. Observe that the notion of predecessor provides a function $p:\X\to\X$ and for every $n\ge1$ the composition $p^n$ gives the $n$-th predecessor of every vertex, while the map $p^0$ denotes the identity map. The \textit{sector} $U_x^{\omega}$ is the infinite subset of $\X$ defined by
\begin{equation*}
    U_x^{\omega}:=\{y\in\X: [x,\omega)\subseteq[y,\omega)\}
\end{equation*}
and $x$ is the \textit{generator} of $U_x^{\omega}$.
Roughly speaking, $U_x^{\omega}$ is the set containing $x$ and all the vertices lying below $x$, namely all its descendants with respect to $\omega$. We will denote by $\dot U_x^{\omega}$ the set $U_x^{\omega}\setminus \{x\}$. Many of these notions are represented in Figure \ref{f1}.

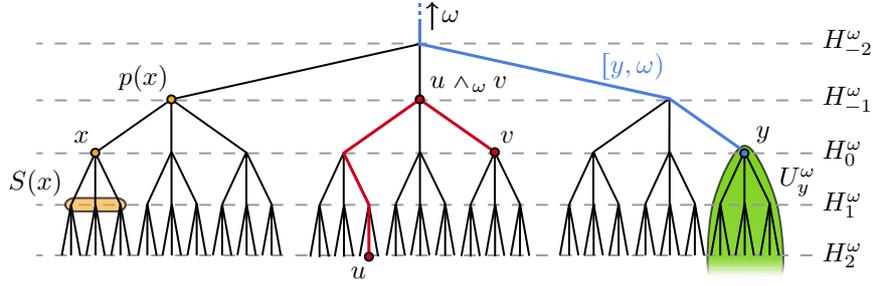
\begin{figure}[h]
    \centering

    \tikzset {_f4ck9b8xy/.code = {\pgfsetadditionalshadetransform{ \pgftransformshift{\pgfpoint{0 bp } { 0 bp }  }  \pgftransformrotate{-90 }  \pgftransformscale{2 }  }}}
\pgfdeclarehorizontalshading{_3rbd5emas}{150bp}{rgb(0bp)=(1,1,1);
rgb(38.57142857142857bp)=(1,1,1);
rgb(56.875bp)=(1,1,1);
rgb(59.91071428571429bp)=(1,1,1);
rgb(100bp)=(1,1,1)}
\tikzset{_ebo4vhymt/.code = {\pgfsetadditionalshadetransform{\pgftransformshift{\pgfpoint{0 bp } { 0 bp }  }  \pgftransformrotate{-90 }  \pgftransformscale{2 } }}}
\pgfdeclarehorizontalshading{_ueaki8ycq} {150bp} {color(0bp)=(transparent!100);
color(38.57142857142857bp)=(transparent!100);
color(56.875bp)=(transparent!90);
color(59.91071428571429bp)=(transparent!0);
color(100bp)=(transparent!0) } 
\pgfdeclarefading{_idbmzd2j8}{\tikz \fill[shading=_ueaki8ycq,_ebo4vhymt] (0,0) rectangle (50bp,50bp); } 
\tikzset{every picture/.style={line width=0.75pt}} 

\begin{tikzpicture}[x=0.75pt,y=0.75pt,scale=0.88,yscale=-1,xscale=-1]
 \draw  [color={rgb, 255:red, 0; green, 0; blue, 0 }  ,draw opacity=0.7 ][fill={rgb, 255:red, 126; green, 211; blue, 33 }  ,fill opacity=1 ] (74,200.5) .. controls (78,200) and (112,202) .. (114,200) .. controls (116,198) and (114.82,172.6) .. (112.32,161.11) .. controls (109.82,149.61) and (100.42,126.6) .. (94.02,126.2) .. controls (87.62,125.8) and (78.02,149.8) .. (75.3,161.11) .. controls (72.59,172.41) and (70,201) .. (74,200.5) -- cycle ;
\draw  [draw opacity=0][shading=_3rbd5emas,_f4ck9b8xy,path fading= _idbmzd2j8 ,fading transform={xshift=2}] (66.88,124.21) -- (118.88,124.21) -- (118.88,207.5) -- (66.88,207.5) -- cycle ;
	\draw [color={rgb, 255:red, 155; green, 155; blue, 155 }  ,draw opacity=1 ] [dash pattern={on 4.5pt off 4.5pt}]  (55.5,160)  -- (495.5,160) ;
	\draw [color={rgb, 255:red, 155; green, 155; blue, 155 }  ,draw opacity=1 ] [dash pattern={on 4.5pt off 4.5pt}]  (55.5,188.83) -- (495.5,188.83) ;
	\draw [color={rgb, 255:red, 155; green, 155; blue, 155 }  ,draw opacity=1 ] [dash pattern={on 4.5pt off 4.5pt}]  (55.5,130.5)  -- (495.5,130.5) ;
	\draw [color={rgb, 255:red, 155; green, 155; blue, 155 }  ,draw opacity=1 ] [dash pattern={on 4.5pt off 4.5pt}]  (55.5,67.75)  -- (495.5,67.75) ;
	\draw [color={rgb, 255:red, 155; green, 155; blue, 155 }  ,draw opacity=1 ] [dash pattern={on 4.5pt off 4.5pt}]  (55.5,100.25) -- (495.5,100.25) ;



	\draw  [color=black, opacity=0.7][fill={rgb, 255:red, 245; green, 166; blue, 35 }  ,fill opacity=0.6 ] (450.32,156.08) -- (473.61,156.08) .. controls (476.63,156.08) and (479.08,157.79) .. (479.08,159.9) .. controls (479.08,162.01) and (476.63,163.73) .. (473.61,163.73) -- (450.32,163.73) .. controls (447.3,163.73) and (444.84,162.01) .. (444.84,159.9) .. controls (444.84,157.79) and (447.3,156.08) .. (450.32,156.08) -- cycle ;
	\draw    (312.19,189.4) -- (306.95,159.9) ;
	\draw    (301.72,189.4) -- (306.95,159.9) ;
	
	\draw    (418.97,99.7) -- (277.92,68) ;
	\draw    (179.69,129.8) -- (136.55,99.3) ;
	\draw    (179.69,159.55) -- (179.69,129.8) ;
	\draw    (193.4,159.7) -- (179.69,129.8) ;
	\draw    (165.26,159.75) -- (179.69,129.8) ;
	\draw    (193.4,188.83) -- (193.4,159.7) ;
	\draw    (198.64,188.83) -- (193.4,159.7) ;
	\draw    (188.17,188.83) -- (193.4,159.7) ;
	\draw    (179.54,188.83) -- (179.54,159.5) ;
	\draw    (184.78,188.83) -- (179.54,159.5) ;
	\draw    (174.31,188.83) -- (179.54,159.5) ;
	\draw    (165.4,188.83) -- (165.4,159.5) ;
	\draw    (170.63,188.83) -- (165.4,159.5) ;
	\draw    (160.17,188.83) -- (165.4,159.5) ;
	\draw    (136.45,159.26) -- (136.45,129.51) ;
	\draw    (150.17,159.41) -- (136.45,129.51) ;
	\draw    (122.02,159.46) -- (136.45,129.51) ;
	\draw    (150.17,188.83) -- (150.17,159.41) ;
	\draw    (155.4,188.83) -- (150.17,159.41) ;
	\draw    (144.93,188.83) -- (150.17,159.41) ;
	\draw    (136.31,188.83) -- (136.31,159.21) ;
	\draw    (141.54,188.83) -- (136.31,159.21) ;
	\draw    (131.07,188.83) -- (136.31,159.21) ;
	\draw    (122.16,188.83) -- (122.16,159.21) ;
	\draw    (127.4,188.83) -- (122.16,159.21) ;
	\draw    (116.93,188.83) -- (122.16,159.21) ;
	\draw    (94.02,159.26) -- (94.02,129.51) ;
	\draw    (107.74,159.41) -- (94.02,129.51) ;
	\draw    (79.59,159.46) -- (94.02,129.51) ;
	\draw    (107.74,188.83) -- (107.74,159.41) ;
	\draw    (112.97,188.83) -- (107.74,159.41) ;
	\draw    (102.5,188.83) -- (107.74,159.41) ;
	\draw    (93.88,188.83) -- (93.88,159.21) ;
	\draw    (99.11,188.83) -- (93.88,159.21) ;
	\draw    (88.64,188.83) -- (93.88,159.21) ;
	\draw    (79.73,188.83) -- (79.73,159.21) ;
	\draw    (84.97,188.83) -- (79.73,159.21) ;
	\draw    (74.5,188.83) -- (79.73,159.21) ;
	\draw    (418.97,129.63) -- (418.97,99.7) ;
	\draw    (376.3,130.2) -- (418.97,99.7) ;
	\draw    (462.11,130.2) -- (418.97,99.7) ;
	\draw    (462.11,159.95) -- (462.11,130.2) ;
	\draw    (475.82,160.1) -- (462.11,130.2) ;
	\draw    (447.68,160.15) -- (462.11,130.2) ;
	\draw    (475.82,188.83) -- (475.82,160.1) ;
	\draw    (481.06,188.83) -- (475.82,160.1) ;
	\draw    (470.59,188.83) -- (475.82,160.1) ;
	\draw    (461.96,188.83) -- (461.96,159.9) ;
	\draw    (467.2,188.83) -- (461.96,159.9) ;
	\draw    (456.73,188.83) -- (461.96,159.9) ;
	\draw    (447.82,188.83) -- (447.82,159.9) ;
	\draw    (453.05,188.83) -- (447.82,159.9) ;
	\draw    (442.59,188.83) -- (447.82,159.9) ;
	\draw    (418.87,159.38) -- (418.87,129.63) ;
	\draw    (432.59,159.53) -- (418.87,129.63) ;
	\draw    (404.44,159.58) -- (418.87,129.63) ;
	\draw    (432.59,188.83) -- (432.59,159.53) ;
	\draw    (437.82,188.83) -- (432.59,159.53) ;
	\draw    (427.35,188.83) -- (432.59,159.53) ;
	\draw    (418.73,188.83) -- (418.73,159.33) ;
	\draw    (423.96,188.83) -- (418.73,159.33) ;
	\draw    (413.49,188.83) -- (418.73,159.33) ;
	\draw    (404.58,188.83) -- (404.58,159.33) ;
	\draw    (409.82,188.83) -- (404.58,159.33) ;
	\draw    (399.35,188.83) -- (404.58,159.33) ;
	\draw    (376.44,159.66) -- (376.44,129.91) ;
	\draw    (390.16,159.81) -- (376.44,129.91) ;
	\draw    (362.01,159.86) -- (376.44,129.91) ;
	\draw    (390.16,188.83) -- (390.16,159.81) ;
	\draw    (395.39,188.83) -- (390.16,159.81) ;
	\draw    (384.92,188.83) -- (390.16,159.81) ;
	\draw    (376.3,188.83) -- (376.3,159.61) ;
	\draw    (381.53,188.83) -- (376.3,159.61) ;
	\draw    (371.06,188.83) -- (376.3,159.61) ;
	\draw    (362.15,188.83) -- (362.15,159.61) ;
	\draw    (367.39,188.83) -- (362.15,159.61) ;
	\draw    (356.92,188.83) -- (362.15,159.61) ;
	\draw    (278.1,129.63) -- (278.1,99.7) ;
	\draw [color={rgb, 255:red, 208; green, 2; blue, 27 }  ,draw opacity=1,line width=0.4 mm ]   (234.97,130.2) -- (278.1,99.7) ;
	\draw [color={rgb, 255:red, 208; green, 2; blue, 27 }  ,draw opacity=1,line width=0.4 mm ]   (321.44,130.5) -- (278.1,99.7) ;
	\draw    (321.24,159.95) -- (321.24,130.2) ;
	\draw    (334.96,160.1) -- (321.24,130.2) ;
	\draw [color={rgb, 255:red, 208; green, 2; blue, 27 }  ,draw opacity=1,line width=0.4 mm ]   (306.81,160.15) -- (321.44,130) ;
	\draw    (334.96,189.85) -- (334.96,160.1) ;
	\draw    (340.19,189.6) -- (334.96,160.1) ;
	\draw    (329.72,189.6) -- (334.96,160.1) ;
	\draw    (321.1,189.65) -- (321.1,159.9) ;
	\draw    (326.33,189.4) -- (321.1,159.9) ;
	\draw    (315.86,189.4) -- (321.1,159.9) ;
	\draw [color={rgb, 255:red, 208; green, 2; blue, 27}  ,draw opacity=1,line width=0.4 mm ]   (306.95,189.65) -- (306.95,159.9) ;
	\draw    (278,159.53) -- (278,129.63) ;
	\draw    (291.72,159.53) -- (278,129.63) ;
	\draw    (263.57,159.58) -- (278,129.63) ;
	\draw    (291.72,189.28) -- (291.72,159.53) ;
	\draw    (296.95,189.03) -- (291.72,159.53) ;
	\draw    (286.49,189.03) -- (291.72,159.53) ;
	\draw    (277.86,189.08) -- (277.86,159.33) ;
	\draw    (283.09,188.83) -- (277.86,159.33) ;
	\draw    (272.63,188.83) -- (277.86,159.33) ;
	\draw    (263.72,189.08) -- (263.72,159.33) ;
	\draw    (268.95,188.83) -- (263.72,159.33) ;
	\draw    (258.48,188.83) -- (263.72,159.33) ;
	\draw    (235.57,159.66) -- (235.57,129.91) ;
	\draw    (249.29,159.81) -- (235.57,129.91) ;
	\draw    (221.15,159.86) -- (235.57,129.91) ;
	\draw    (249.29,189.56) -- (249.29,159.81) ;
	\draw    (254.52,189.31) -- (249.29,159.81) ;
	\draw    (244.06,189.31) -- (249.29,159.81) ;
	\draw    (235.43,189.36) -- (235.43,159.61) ;
	\draw    (240.66,189.11) -- (235.43,159.61) ;
	\draw    (230.2,189.11) -- (235.43,159.61) ;
	\draw    (221.29,189.36) -- (221.29,159.61) ;
	\draw    (226.52,189.11) -- (221.29,159.61) ;
	\draw    (216.05,189.11) -- (221.29,159.61) ;
	
	\draw    (136.65,131.05) -- (136.55,99.3) ;

	\draw    (278.1,99.7) -- (278.1,68) ;

	\draw [color={rgb, 255:red, 74; green, 125; blue, 226 }  ,draw opacity=1,line width=0.4 mm ]   (136.35,99.3) -- (278.1,68) ;
	\draw [color={rgb, 255:red, 74; green, 125; blue, 226 }  ,draw opacity=1,line width=0.4 mm ]   (94.02,129.51) -- (136.55,99.3) ;
	\draw [color={rgb, 255:red, 74; green, 125; blue, 226 }  ,draw opacity=1,line width=0.4 mm ] [dash pattern={on 1.5pt off 1.5pt on 1.5pt off 1.5pt}]  (278.1,56) -- (278.1,42) ;
	\draw [color={rgb, 255:red, 74; green, 125; blue, 226 }  ,draw opacity=1,line width=0.4 mm ]   (278.1,68.8) -- (278.1,56) ;
	\draw  [color=black  ,draw opacity=1 ][fill={rgb, 255:red, 208; green, 2; blue, 27 }  ,fill opacity=1 ] (233.26,129.91) .. controls (233.26,128.64) and (234.29,127.6) .. (235.57,127.6) .. controls (236.85,127.6) and (237.89,128.64) .. (237.89,129.91) .. controls (237.89,131.19) and (236.85,132.23) .. (235.57,132.23) .. controls (234.29,132.23) and (233.26,131.19) .. (233.26,129.91) -- cycle ;
	\draw  [color=black  ,draw opacity=1 ][fill={rgb, 255:red, 208; green, 2; blue, 27 }  ,fill opacity=1 ] (304.64,189.65) .. controls (304.64,188.37) and (305.68,187.34) .. (306.95,187.34) .. controls (308.23,187.34) and (309.27,188.37) .. (309.27,189.65) .. controls (309.27,190.93) and (308.23,191.96) .. (306.95,191.96) .. controls (305.68,191.96) and (304.64,190.93) .. (304.64,189.65) -- cycle ;
	\draw  [color=black  ,draw opacity=1 ][fill={rgb, 255:red, 208; green, 2; blue, 27 }  ,fill opacity=1 ] (275.79,99.7) .. controls (275.79,98.42) and (276.82,97.39) .. (278.1,97.39) .. controls (279.38,97.39) and (280.42,98.42) .. (280.42,99.7) .. controls (280.42,100.98) and (279.38,102.01) .. (278.1,102.01) .. controls (276.82,102.01) and (275.79,100.98) .. (275.79,99.7) -- cycle ;
	\draw  [color=black  ,draw opacity=1 ][fill={rgb, 255:red, 245; green, 166; blue, 35 }  ,fill opacity=1 ] (459.79,130.2) .. controls (459.79,128.92) and (460.83,127.89) .. (462.11,127.89) .. controls (463.38,127.89) and (464.42,128.92) .. (464.42,130.2) .. controls (464.42,131.48) and (463.38,132.51) .. (462.11,132.51) .. controls (460.83,132.51) and (459.79,131.48) .. (459.79,130.2) -- cycle ;
	\draw  [color={black}  ,draw opacity=1 ][fill={rgb, 255:red, 245; green, 166; blue, 35 }  ,fill opacity=1 ] (416.65,99.7) .. controls (416.65,98.42) and (417.69,97.39) .. (418.97,97.39) .. controls (420.25,97.39) and (421.28,98.42) .. (421.28,99.7) .. controls (421.28,100.98) and (420.25,102.01) .. (418.97,102.01) .. controls (417.69,102.01) and (416.65,100.98) .. (416.65,99.7) -- cycle ;
	\draw  [color=black  ,draw opacity=1 ][fill={rgb, 255:red, 74; green, 125; blue, 226 }  ,fill opacity=1 ] (91.7,130.51) .. controls (91.7,129.24) and (92.74,128.2) .. (94.02,128.2) .. controls (95.3,128.2) and (96.33,129.24) .. (96.33,130.51) .. controls (96.33,131.79) and (95.3,132.83) .. (94.02,132.83) .. controls (92.74,132.83) and (91.7,131.79) .. (91.7,130.51) -- cycle ;
	
	\node at (415,101) [anchor=south east] {$p(x)$};
	\node at (460,130) [anchor=south east] {$x$};
	\node at (475,160.5) [anchor=south east] {$S(x)$};
	\node at (93,130) [anchor=south west] {$y$};
	\node at (80,160.5) [anchor=south west] {$U^\omega_y$};
	\node at (238,130) [anchor=south west] {$v$};
	\node at (278,101) [anchor=south west] {$u\wedge_\omega v$};
	\node at (303,190) [anchor=north east] {$u$};
	\node at (55.5,188.83) [anchor=west] {$H^\omega_2$};
	\node at (55.5,160) [anchor=west] {$H^\omega_1$};
	\node at (55.5,130.5) [anchor=west] {$H^\omega_0$};
	\node at (55.5,100.25) [anchor=west] {$H^\omega_{-1}$};
	\node at (55.5,67.75) [anchor=west] {$H^\omega_{-2}$};
    \node at (181,81) [color={rgb, 255:red, 74; green, 125; blue, 226 },anchor=west] {$[y,\omega)$};
 
	\draw[->] (271.1,60) -- (271.1,44); 
	\node at (271.1,52) [anchor=west] {$\omega$};
	
\end{tikzpicture}

   \begin{minipage}{12cm}
    \caption{\small This figure represents a portion of the 3-homogeneous tree. On the left, the neighbours of $x$ are split in the predecessor and the successors. In the middle, the confluent of two vertices. On the right, the green area is the first portion of the sector $U_y^\omega$ of $y$ and the blue path is $[y,\omega)$. Finally, the figure is foliated by (a portion of) some of the horocycles tangent to~$\omega$ \label{f1}}
\end{minipage}
\end{figure}

Finally, the notion of confluent allows us to define another metric on $\X$, the \textit{horocyclic Gromov distance} $\rho_{\omega}:\X\times\X\to\R$, defined by
\begin{equation*}
    \rho_{\omega}(x,y)=
    \begin{cases}
        1 &\text{ if } x=y, \\
        e^{-\hi{x\wedge_{\omega} y}_\omega} &\text{ if } x\ne y,
    \end{cases}
\end{equation*}
see~\cite{g} and~\cite{arsw} for more details about Gromov metric in general and in discrete settings, respectively.

From now on $\omega\in\Omega$ will be fixed and we will omit it from all the notations, namely we will write $x\wedge y$, $\Ho{k}$, $\HBo{n}$, $\hi{x}$, $U_x$, $\dot U_x$, $\rho$.

\subsection{Harmonic functions on $\X$}

\begin{defin}
    Let $f$ be a complex valued function on $\X$. The \textit{combinatorial Laplacian} of $f$ is defined as
    \begin{equation}\label{cl}
        \Delta f(x):=\frac{1}{q+1}\sum_{y\sim x} f(y)-f(x), \quad x\in\X.
    \end{equation}
    The function $f$ is said to be \textit{harmonic in $x$} if $\Delta f(x)=0$, namely if it satisfies the mean property
    \begin{equation*}
        f(x)=\frac{1}{q+1}\sum_{y\sim x} f(y).
    \end{equation*}
    We will say that $f$ is \textit{harmonic on $Y$}$\subseteq\X$ if it is harmonic in $x$ for every $x\in Y$ and that it is \textit{harmonic} if it is harmonic on the whole of $\X$.
\end{defin}
Notice that the notion of harmonicity of a function in a point or on a proper subset of $\X$ involves its values on a larger subset of $\X$. Furthermore, it is clear that the value of a harmonic function at a point $x$ affects the values in the neighbours of $x$, which affect the values in their neighbours and so on, implying a certain stiffness of the function. Indeed, the next result, which is an analogue of Lemma 4.1 in \cite{ccps} and Proposition 2 in \cite{dmmv}, shows that the harmonicity of a function on a sector implies that certain averages (see \eqref{mos}) of the function depend on the values at the generator and at its predecessor.

\begin{lemma}\label{p1}
Let $y\in\X$. If $f:\X\rightarrow\mathbb{C}$ is harmonic on $U_y$, then for all $n\in\N$
\begin{equation}\label{hc}
    \sum_{\substack{ x\in U_y\cap\Ho{\hi{y}+n}}} f(x)
    = \Bigl(\sum_{j=0}^{n} q^j\Bigr)f(y)
    -\Bigl(\sum_{j=0}^{n-1} q^j\Bigr)f(p(y)).
\end{equation}
Conversely, if $f$ is constant on $U_y\cap \Ho{\hi{y}+k}$ and satisfies \eqref{hc} for all $k\in\N$, then $f$ is harmonic on $U_y$.
\begin{proof}
Let $A_n:=U_y\cap\Ho{\hi{y}+n}$ for every $n\in\N$. It is clear that $A_n$ is the set of points in $U_y$ which have distance $n$ from $y$ and that $\# A_n=q^n$. We set $a_{-1}=f(p(y))$ and for every $n\in\N$,
\begin{equation}\label{mos}
    a_{n}
    =\frac{1}{q^n}\sum_{x\in A_n} f(x).
\end{equation}
Since, for every $n\ge1$, $f$ is harmonic on $A_n\subseteq U_y$ we have that
\begin{align*}
    q^n a_n
    &=\sum_{x\in A_n} f(x)=\sum_{x\in A_n} \frac{1}{q+1} \Bigl(f(p(x))+\sum_{z\in S(x)} f(z)\Bigr) \\
    &= \frac{1}{q+1} \Bigl( \sum_{x\in A_n} f(p(x)) +\sum_{x\in A_n} \sum_{z\in S(x)} f(z)\Bigr) \\
    &= \frac{1}{q+1} \Bigl( q \sum_{x\in A_{n-1}} f(x) +\sum_{z\in A_{n+1}} f(z)\Bigr) \\
    &= \frac{1}{q+1} (qq^{n-1} a_{n-1}+q^{n+1}a_{n+1}).
\end{align*}
Notice that $a_0=f(y)$ and the harmonicity of $f$ in $y$ reads
\begin{equation*}
    a_0=\frac{1}{q+1}(a_{-1}+q a_1).
\end{equation*}
Thus we obtained that for all $n\in\N$,
\begin{equation}\label{a1}
    q^{n+1}a_{n+1}=(q+1)q^na_n-q^na_{n-1}.
\end{equation}
Our claim is that for all $m\in\N$
\begin{equation}\label{a2}
    q^ma_m
    =\Bigl(\sum_{j=0}^m q^j\Bigr)a_0
    -\Bigl(\sum_{j=0}^{m-1}q^j\Bigr)a_{-1}.
\end{equation}
We proceed by induction on $m$. The case $m=0$ is trivial and the case $m=1$ follows from \eqref{a1} with $n=0$. Now assume $m\ge1$ and that \eqref{a2} holds for every $i\le m$, then by~\eqref{a1}
\begin{align*}
    &q^{m+1}a_{m+1}
    =(q+1)q^ma_m-qq^{m-1}a_{m-1}\\
    &=(q+1)\Bigl(\sum_{j=0}^m q^j\Bigr)a_0-(q+1)\Bigl(\sum_{j=0}^{m-1}q^j\Bigr)a_{-1}
    -q\Bigl(\sum_{j=0}^{m-1} q^j\Bigr)a_0+ q\Bigl(\sum_{j=0}^{m-2}q^j\Bigr)a_{-1} \\
    &= a_0\Bigl(\sum_{j=0}^m q^{j+1}+\sum_{j=0}^m q^j-\sum_{j=0}^{m-1} q^{j+1}\Bigr)
    -a_{-1}\Bigl(\sum_{j=0}^{m-1} q^{j+1}+\sum_{j=0}^{m-1} q^j-\sum_{j=0}^{m-2} q^{j+1}\Bigr)\\
    &= a_0\Bigl(\sum_{j=0}^{m+1} q^j\Bigr)-a_{-1}\Bigl(\sum_{j=0}^m q^j\Bigr).
\end{align*}
Now if $f:\X\rightarrow\C$ satisfies \eqref{hc} and is constant on $A_k$ for all $k\in\N$, say $f(x)=f_k$ for all $x\in A_k$, then \eqref{hc} reads
\begin{equation*}
    q^kf(x)=q^kf_k
    = \Bigl(\sum_{j=0}^{k} q^j\Bigr)f(y)
    -\Bigl(\sum_{j=0}^{k-1} q^j\Bigr)f(p(y))
\end{equation*}
and proceeding backwards in the first part of the proof gives that $f$ is harmonic on $U_y$.
\end{proof}
\end{lemma}

Now we consider a function $g:\X\rightarrow\C$ which is harmonic on the horoball $\HBo{n}$, $n\in\Z$. There are infinitely many ways to extend $g|_{\HBo{n+1}}$ to a harmonic function on $\X$. Lemma \ref{p1} allows us to define a function $g^H_n:\X\rightarrow\C$ which coincides with $g$ on $\HBo{n+1}$, is harmonic on the whole of $\X$ and is constant on $U_y\cap\Ho{k}$ for all $y\in\Ho{n+1}$ and $k\ge n+1$.
Indeed, notice that for every $x\in\X\setminus\HBo{n+1}$ there exists a unique $y\in\Ho{n+1}$ such that $x\in U_y$, namely $y=p^{\hi{x}-n-1}(x)$. In order to have $g^H_n$ constant on $U_y\cap \Ho{\hi{x}}$, by \eqref{hc} we want $g^H_n$ to satisfy
\begin{equation*}
    q^{\hi{x}-n-1}g^H_n(x)=\Bigl(\sum_{j=0}^{\hi{x}-n-1} q^j\Bigr)g(y)-\Bigl(\sum_{j=0}^{\hi{x}-n-2} q^j\Bigr)g(p(y))
\end{equation*}
which gives
\begin{equation*}
    g^H_n(x)
    =\Bigl(\sum_{j=0}^{\hi{x}-n-1} q^{-j}\Bigr)g(y)
    -\Bigl(\sum_{j=1}^{\hi{x}-n-1} q^{-j}\Bigr)g(p(y)).
\end{equation*}
Hence the function $g^H_n$ is given by
\begin{equation*}
    g^H_n(x)=\begin{dcases} g(x), &\mbox{if } x\in\HBo{n+1}, \\ \Bigl(\sum_{j=0}^{\hi{x}-n-1} q^{-j}\Bigr)g(y)-\Bigl(\sum_{j=1}^{\hi{x}-n-1} q^{-j}\Bigr)g(p(y)), &\mbox{if } x\in\X\setminus\HBo{n+1},
    \end{dcases}
\end{equation*}
where $y=p^{\hi{x}-n-1}(x)$. We will call $g^H_n$ the \textit{harmonic extension} of $g$ below the $n$-th horoball.

\section{Horocyclic Bergman measures and functions spaces}\label{sec.hm}

In this section we are interested in studying horocyclic Bergman measures in order to define Bergman spaces of harmonic functions as well as atomic Hardy spaces and bounded mean oscillation spaces.

\subsection{Horocyclic Bergman measures}

The idea is to consider a measure whose Radon-Nikodym derivative with respect to the counting measure at a given vertex only depends on the horocyclic index. Inspired by the radial measures considered in \cite{ccps}, the so-called \textit{reference measures}, and in particular by those studied in \cite{dmmv}, we introduce the family of measures $\sigma_{\alpha}$ given by
\begin{equation*}
    \sigma_{\alpha}(x)
    =q^{-\alpha\hi{x}}, \quad x\in\X,
\end{equation*}
for every $\alpha>1$, that we shall call \textit{horocyclic Bergman measures}. Observe that our measures are no longer radial with respect to some origin $o$ but are constant on the horocycles. Furthermore, they are not finite on the whole of $\X$, although they have the nice property to be finite on the sectors $U_x$ for all $x\in\X$. Indeed, notice that for every $x\in\X$ the sector $U_x$ can be seen as the disjoint union of the sets $U_x\cap \Ho{\hi{x}+i}$ for $i\in\N$. Since $\#U_x\cap \Ho{\hi{x}+i}=q^i$, and $\sigma_{\alpha}$ is constant on $\Ho{\hi{x}+i}$ we have
\begin{align}\label{ms}
    \notag\sigma_{\alpha}(U_x)
    &=\sum_{y\in U_x} \sigma_{\alpha}(y)
    =\sum_{i=0}^{+\infty} \, \sum_{y\in U_x\cap \Ho{\hi{x}+i}} \sigma_{\alpha}(y) \\
    &=\sum_{i=0}^{+\infty} q^iq^{-\alpha(\hi{x}+i)}
    =q^{-\alpha\hi{x}}\sum_{i=0}^{+\infty} q^{i(1-\alpha)}
    =\frac{q^{-\alpha\hi{x}}}{1-q^{1-\alpha}}.
\end{align}
Remark that \eqref{ms} holds true because we ask $\alpha$ to be greater than $1$. Furthermore, this shows that the measure of a sector is comparable to the measure of its generator with a constant depending only on $\alpha$. In the context of radial measures this is the \textit{optimality} notion introduced in \cite{ccps}.

We introduce the notion of \textit{doubling} space, that will be useful later, and show that $\X$ with the measure $\sigma_{\alpha}$ is not doubling with respect to the counting distance, while it is doubling with respect to the Gromov distance.

\begin{defin}
    A measure metric space $(X,d,\mu)$ is \textit{doubling} if there exists a constant $C>0$ such that
    \begin{equation*}
        \mu(B(x,2r))\le C\mu(B(x,r))
    \end{equation*}
    for all $x\in X$ and $r>0$.
\end{defin}

\begin{prop}
    For every $\alpha>1$, the measure metric space $(\X,d,\sigma_{\alpha})$ is not doubling.
    \begin{proof}
        Fix a vertex $y\in\Ho{0}$. For all $n\in\N$ select $v_n\in U_y\cap\Ho{2n}$. We consider the ball $B_d(v_n,n)$ and notice that $B_d(v_n,n)\subseteq\X\setminus\HBo{n-1}$, thus
        \begin{equation*}
            \max\{\sigma_{\alpha}(x):x\in B_d(v_n,n)\}=q^{-\alpha n}.
        \end{equation*}
        It is easy to observe that $\# S_d(x,k)=(q+1)q^{k-1}$ for every $x\in\X$ and integer $k>0$.
        Therefore, we obtain
        \begin{align*}
            \sigma_{\alpha}(B_d(v_n,n))
            &=\sum_{x\in B_d(v_n,n)}\sigma_{\alpha}(x)
            \le q^{-\alpha n} \,\#B_d(v_n,n)
            =q^{-\alpha n} \sum_{i=0}^{n} \#S_d(v_n,i) \\
            &=q^{-\alpha n} \Bigl(1+\sum_{i=1}^n (q+1)q^{i-1}\Bigr)
            \lesssim q^{n(1-\alpha)}.
        \end{align*}
        It follows that
        \begin{equation*}
        \frac{\sigma_{\alpha}(B_d(v_n,2n))}{\sigma_{\alpha}(B_d(v_n,n))} \gtrsim  \frac{\sigma_{\alpha}(y)}{q^{n(1-\alpha)}}
        = q^{n(\alpha-1)} \longrightarrow +\infty
        \end{equation*}
        so that $(\X,d,\sigma_{\alpha})$ is not doubling for any $\alpha>1$.
    \end{proof}
\end{prop}

\begin{rem}\label{gb}
    For what follows we need to understand the structure of balls in the Gromov metric $\rho$, that is
    \begin{equation*}
        B_{\rho}(x,r)=\{y\in\X:\rho(x,y)\le r\}.
    \end{equation*}
    We show that $B_{\rho}(x,r)$ is either $\{x\}$ or a sector containing $x$. Indeed, for $x,y\in\X$, $x\ne y$ and $r>0$ we have that
    \begin{equation*}
        \rho(x,y)=e^{-\hi{x\wedge y}}\le r
        \iff
        -\log(r)\le\hi{x\wedge y}.
    \end{equation*}
    Notice that $x\wedge y\in[x,\omega)$, so  $\hi{x\wedge y}\le\hi{x}$ and
        \begin{enumerate}[i)]
            \item if $\lfloor -\log(r)\rfloor>\hi{x}$, then $B_{\rho}(x,r)=\{x\}$;
            \item if $\lfloor -\log(r)\rfloor\le\hi{x}$, then
            $$B_{\rho}(x,r)=\{y\in\X:\lfloor -\log(r)\rfloor\le\hi{x\wedge y}\le\hi{x}\}
            =U_{p^{\hi{x}-\lfloor -\log(r)\rfloor}(x)}$$ since $x\wedge y$ must lie in $[x,p^{\hi{x}-\lfloor -\log(r)\rfloor}(x)]$. In particular $B_{\rho}(x,r)=U_x$ when $\lfloor -\log(r)\rfloor=\hi{x}$.
        \end{enumerate}
        In conclusion, we can state that
        \begin{equation*}
            \{B_{\rho}(x,r):r>0\}=\{x\}\cup\{U_y:y\in\X,x\in U_y\}.
        \end{equation*}
\end{rem}

\begin{prop}\label{gmd}
    For every $\alpha>1$, the measure metric space $(\X,\rho,\sigma_{\alpha})$ is doubling with doubling constant
    \begin{equation*}
        k_{\alpha}=\max\Bigl\{q^{\alpha},\frac{1}{1-q^{1-\alpha}}\Bigr\}.
    \end{equation*}
    \begin{proof}
        Let $x\in\X$ and $r>0$. We observe that
        \begin{equation*}
            \lfloor-\log(2r)\rfloor
            =\lfloor -\log(2)-\log(r)\rfloor\in\left\{\lfloor -\log(r)\rfloor, \lfloor -\log(r)\rfloor-1\right\}.
        \end{equation*}
        For this reason, we have that if $B_{\rho}(x,r)=\{x\}$ then $B_{\rho}(x,2r)$ can only be $\{x\}$ or $U_x$ while if $B_{\rho}(x,r)=U_y$ then $B_{\rho}(x,2r)$ can only be $U_y$ or $U_{p(y)}$. Thus it is sufficient to check that
        \begin{equation*}
            \sup_{y\in\X} \frac{\sigma_{\alpha}(U_y)}{\sigma_{\alpha}(y)}<+\infty, \quad
            \sup_{y\in\X} \frac{\sigma_{\alpha}(U_{p(y)})}{\sigma_{\alpha}(U_y)}<+\infty.
        \end{equation*}
        Both facts follow by \eqref{ms}, indeed
        \begin{equation*}
            \frac{\sigma_{\alpha}(U_y)}{\sigma_{\alpha}(\{y\})}
            =\frac{q^{-\alpha\hi{y}}\frac{1}{1-q^{1-\alpha}}}{q^{-\alpha\hi{y}}}
            =\frac{1}{1-q^{1-\alpha}}
        \end{equation*}
        and
        \begin{equation*}
            \frac{\sigma_{\alpha}(U_{p(y)})}{\sigma_{\alpha}(U_y)}
            =\frac{q^{-\alpha(\hi{y}-1)}\frac{1}{1-q^{1-\alpha}}}{q^{-\alpha\hi{y}}\frac{1}{1-q^{1-\alpha}}}
            =q^{\alpha}.
        \end{equation*}
        We conclude that $(\X,\rho,\sigma_{\alpha})$ is doubling for all $\alpha>1$ with doubling constant
        \begin{equation*}
            k_{\alpha}
            =\max\Bigl\{q^{\alpha},\frac{1}{1-q^{1-\alpha}}\Bigr\}.
        \end{equation*}
    \end{proof}
\end{prop}

\subsection{Horocyclic harmonic Bergman spaces}

In what follows we study harmonic Bergman spaces with respect to the measures introduced above, in analogy with \cite{ccps,dmmv}.

\begin{defin}
For $p\in[1,+\infty)$ and $\alpha>1$ the \textit{horocyclic harmonic Bergman space} $\B{p}$ is the space of harmonic functions $f:\X\to\C$ such that
\begin{equation*}
    \|f\|_{\B{p}}
    :=\Bigl(\sum_{x\in\X}|f(x)|^p\sigma_{\alpha}(x)\Bigr)^{\frac{1}{p}}
    < +\infty.
\end{equation*}
\end{defin}

From now on we fix $\alpha>1$.

\begin{ex}\label{es}
We explicitly build a function in $\B{p}$. We will use many times this kind of construction in the rest of the work. Fix a vertex $y\in\X$ and let $g:S(y)\to\C$ be such that
\begin{equation*}
    \sum_{x\in S(y)}g(x)=0, \quad g\not\equiv 0.
\end{equation*}
Let $f$ be defined by
\begin{equation*}
    f(x)=
    \begin{dcases}
    0, &\mbox{if } x\in\X\setminus S(y), \\
    g(x), &\mbox{if } x\in S(y).
    \end{dcases}
\end{equation*}
It is clear that $f$ is harmonic on $\HBo{\hi{y}}$, so that we can consider its harmonic extension $f^H_{\hi{y}}$. Since $p^{\hi{x}-\hi{y}}(x)=y$ for every $x\in\dot U_y$ and $f(y)=0$, we have
\begin{equation*}
    f^H_{\hi{y}}(x)=
    \begin{dcases}
    0, &\mbox{if } x\notin \dot U_y, \\
    \Bigl(\sum_{i=0}^{\hi{x}-\hi{y}-1}q^{-i}\Bigr)g(p^{\hi{x}-\hi{y}-1}(x)), &\mbox{if } x\in\dot U_y.
    \end{dcases}
\end{equation*}
Thus $f^H_{\hi{y}}$ is harmonic and is bounded by
\begin{equation*}
    |f^H_{\hi{y}}(x)|\le \max_{z\in S(y)} |g(z)|\, \sum_{i=0}^{+\infty}q^{-i}=
    \max_{z\in S(y)} |g(z)| \, \frac{q}{q-1}<+\infty.
\end{equation*}
Furthermore, since $\supp(f^H_{\hi{y}})\subseteq U_y$ and $\sigma_{\alpha}(U_y)<+\infty$, $f^H_{\hi{y}}\in\B{p}$ for all $p\in[1,+\infty)$ and $\alpha>1$.
\end{ex}

The next result is quite technical and will be a useful tool for many computations. 
\begin{lemma}\label{sp}
Let $n\in\Z$, $y\in\Ho{n}$ and $g:\X\to\C$ be a function harmonic on $\HBo{n}$ and such that $g(x)=0$ for all $x\in\HBo{n+1}$ except on $S(y)$. There exist $b_{n,\alpha}, b_{n,\alpha}'>0$ such that, for every $f\in\B{2}$,
\begin{equation*}
    \langle f,g^H_n\rangle_{\B{2}}
    =\sum_{z\in S(y)}\overline{g(z)}\left(f(z) b_{n,{\alpha}}
    -f(y)b_{n,\alpha}'\right),
\end{equation*}
where in particular
\begin{equation*}
    b_{n,\alpha}
    :=q^{-\alpha n}\sum_{l=0}^{+\infty}q^{-\alpha (l+1)} \Bigl(\sum_{i=0}^{l} q^{-i}\Bigr)\Bigl(\sum_{j=0}^{l} q^j\Bigr),
\end{equation*}
\begin{equation*}
    b_{n,\alpha}'
    :=q^{-\alpha n}\sum_{l=0}^{+\infty}q^{-\alpha (l+1)} \Bigl(\sum_{i=0}^{l} q^{-i}\Bigr)\Bigl(\sum_{j=0}^{l-1} q^j\Bigr).
\end{equation*}
\begin{proof}
    Let $f,n$ and $g$ be as in the statement. It is worth noticing that $g$ has zero mean on $S(y)$ since it is harmonic in $y$ and vanishes on $\HBo{n}$, so that $g^H_n$ belongs to $\B{2}$ by Example \ref{es}. Furthermore, it is clearly supported in $\dot U_y$, hence
    \begin{align*}
        \langle f,g^H_n\rangle_{\B{2}}
        &=\sum_{x\in\dot U_y}f(x)\overline{g_n^H(x)}\sigma_{\alpha}(x)
        =\sum_{k>n} q^{-\alpha k} \sum_{x\in\Ho{k} \cap U_y}f(x)\overline{g_n^H(x)}\\
        &=\sum_{k>n}q^{-\alpha k} \sum_{x\in\Ho{k}\cap U_y} f(x)
        \Bigl[\Bigl(\sum_{i=0}^{k-n-1} q^{-i}\Bigr)\overline{g(z)}
        -\Bigl(\sum_{i=1}^{k-n-1}q^{-i}\Bigr)\overline{g(y)}\Bigr]
    \end{align*}
    where $z=p^{k-n-1}(x)$ for every $x\in U_y\cap\Ho{k}$. Now, since $g(y)=0$, we have
    \begin{align*}
        \langle  f,&g^H_n\rangle_{\mathcal{B}^2_{\alpha}} 
        =\sum_{k>n}q^{-\alpha k} \Bigl(\sum_{i=0}^{k-n-1} q^{-i}\Bigr) \sum_{x\in\Ho{k}\cap U_y} f(x) \overline{g(p^{k-n-1}(x))} \\
        &=\sum_{k>n}q^{-\alpha k} \Bigl(\sum_{i=0}^{k-n-1} q^{-i}\Bigr) \sum_{z\in S(y)}\overline{g(z)}\sum_{x\in\Ho{k}\cap U_z} f(x) \\
        &=\sum_{k>n}q^{-\alpha k} \Bigl(\sum_{i=0}^{k-n-1} q^{-i}\Bigr) \sum_{z\in S(y)}\overline{g(z)}\Bigl(\Bigl(\sum_{j=0}^{k-n-1}q^j\Bigr)f(z)
        -\Bigl(\sum_{j=0}^{k-n-2}q^j\Bigr)f(p(z))\Bigr)\\
        &=\sum_{z\in S(y)}\overline{g(z)}
        \Bigl(f(z)b_{n,\alpha}-f(y)b_{n,\alpha}'\Bigr),
    \end{align*}
    where we applied Lemma \ref{p1} to $f$.
    Also, we have
    \begin{align*}
        b_{n,\alpha}
        :=&\sum_{k>n}q^{-\alpha k} \Bigl(\sum_{i=0}^{k-n-1} q^{-i}\Bigr)\Bigl(\sum_{j=0}^{k-n-1} q^j\Bigr)
        =\sum_{l=0}^{+\infty}q^{-\alpha (l+n+1)} \Bigl(\sum_{i=0}^{l} q^{-i}\Bigr)\Bigl(\sum_{j=0}^{l} q^j\Bigr)\\
        =&q^{-\alpha n}\sum_{l=0}^{+\infty}q^{-\alpha (l+1)} \Bigl(\sum_{i=0}^{l} q^{-i}\Bigr)\Bigl(\sum_{j=0}^{l} q^j\Bigr)
    \end{align*}
    and analogously for $b_{n,\alpha}'$.
\end{proof}
\end{lemma}
By the same proof we can also obtain that for every $g\in\B{2}$ that vanishes in $\HBo{n}$ and every $f\in\B{2}$,
\begin{equation*}
    \langle f,g^H_n\rangle_{\B{2}}
    =\sum_{z\in \Ho{n+1}}\overline{g(z)}\left(f(z) b_{n,{\alpha}}
    -f(p(z))b_{n,\alpha}'\right).
\end{equation*}
From now on we will write $b_n$ and $b_n'$ instead of $b_{n,\alpha}$ and $b_{n,\alpha}'$ since $\alpha$ is fixed.

\vspace{0.3 cm}
We conclude this section by showing that, as we could expect, $\B{2}$ is a Hilbert space. In the following sections, we provide an orthonormal basis and a reproducing kernel for $\B{2}$.

\begin{prop}\label{hs}
    The horocyclic Bergman space $\B{2}$ is a Hilbert space.
    \begin{proof}
        We only have to show that $\B{2}$ is closed in $\Lda$. The combinatorial Laplacian $\Delta$ defined in \eqref{cl} maps $\Lda$ into $\Lda$, indeed for $f\in\Lda$
        \begin{align*}
            \|\Delta f\|^2_{2,\alpha}
            &=\sum_{x\in\X}|\Delta f(x)|^2\sigma_{\alpha}(x)\\
            &=\sum_{x\in\X}\Bigl|\frac{1}{q+1}\sum_{y\sim x}f(y)-f(x)\Bigr|^2\sigma_{\alpha}(x)\\
            &\le \sum_{x\in\X} (q+2) \Bigl(\frac{1}{(q+1)^2}\sum_{y\sim x}|f(y)|^2+|f(x)|^2\Bigr)\sigma_{\alpha}(x)\\
            &=(q+2) \Bigl(\frac{1}{(q+1)^2}\sum_{x\in\X}\sum_{y\sim x}|f(y)|^2\sigma_{\alpha}(x)
            +\sum_{x\in\X}|f(x)|^2\sigma_{\alpha}(x)\Bigr),
        \end{align*}
        where the inequality is given by the equivalence of the norms in $\C^{q+2}$, in particular by
        \begin{equation*}
            \|v\|_1\le\sqrt{q+2}\,\|v\|_2, \quad v\in\C^{q+2}.
        \end{equation*}
        The first term of the last equality can be rewritten as
        \begin{align*}
            \frac{1}{(q+1)^2}&\sum_{x\in\X}\sum_{y\sim x}|f(y)|^2\sigma_{\alpha}(x)\\
            &=\frac{1}{(q+1)^2}\Bigl(\sum_{x\in\X}|f(p(x))|^2\sigma_{\alpha}(p(x))q^{-\alpha}
            +\sum_{x\in\X}\sum_{y\in S(x)}|f(y)|^2\sigma_{\alpha}(y)q^{\alpha}\Bigr)\\
            &=\frac{1}{(q+1)^2} \left(q \|f\|^2_{2,\alpha} q^{-\alpha}
            +\|f\|^2_{2,\alpha}q^{\alpha}\right)\\
            &=\frac{q^{1-\alpha}+q^{\alpha}}{(q+1)^2}\,\|f\|^2_{2,\alpha}.
        \end{align*}
        This gives
        \begin{equation*}
            \|\Delta f\|^2_{2,\alpha}
            \le (q+2)\Bigl(\frac{q^{1-\alpha}+q^{\alpha}}{(q+1)^2}+1\Bigr)\|f\|^2_{2,\alpha}
        \end{equation*}
        so that $\Delta f\in\Lda$ and $\Delta:\Lda\to\Lda$ is continuous. It is clear that $\B{2}=\Delta^{-1}(\{0\})$ and so it is closed in $\Lda$ and it is a Hilbert space.
    \end{proof}
\end{prop}

\subsection{$H^1$ and $BMO$ spaces}\label{subs.hbmo}

We introduce now some suitable versions of the classical atomic Hardy spaces and bounded mean oscillation spaces. As already mentioned, in Section~\ref{sec.br} these spaces will be involved in some limiting results concerning the boundedness of integral operators. For a deeper presentation of these spaces in more general settings we refer to \cite{cmm} and \cite{cw}, while for Hardy and $BMO$ spaces on homogeneous trees, see \cite{atv}, \cite{atv2} and \cite{m}.

Before giving the next definition it is worth noticing that the collection of all the balls with respect to the Gromov distance $\rho$, that is
\begin{equation*}
    \mathcal{D}=\Bigl\{ U_x, \{x\}: x\in\X\Bigr\},
\end{equation*}
plays a similar role, with respect to $\sigma_{\alpha}$, of the dyadic sets in the Euclidean setting. In particular for all $k\in\Z$ we set \begin{equation*}
    \mathcal{D}_k=\Bigl\{U_x:x\in\Ho{k}\Bigr\}\cup\Bigl\{\{y\}:y\in\HBo{k-1}\Bigr\}\subseteq\mathcal{D}.
\end{equation*}
It is clear that for every $k\in\Z$, the family $\mathcal{D}_k$ is a partition of $\X$, $\mathcal{D}_k$ is a refinement of $\mathcal{D}_{k-1}$ and for all $D\in\mathcal{D}_k$ and $D'\in\mathcal{D}_{k-1}$ such that $D\subseteq D'$, we have $\sigma_{\alpha}(D)\approx\sigma_{\alpha}(D')$.
For further details on dyadic sets for trees we refer to \cite{m}.

\begin{defin}
    Let $1<p\le\infty$, a $(1,p)$-atom is a function $a$ such that
    \begin{enumerate}[i)]
        \item $a$ is supported in $D$ for some $D\in\mathcal{D}$;
        \item $\|a\|_{p,\alpha}\le\sigma_{\alpha}(D)^{\frac{1}{p}-1}$ if $p<\infty$ and $\|a\|_{\infty,\alpha}\le\sigma_{\alpha}(D)^{-1}$ if $p=\infty$;
        \item $a$ has vanishing mean on $D$, that is
        \begin{equation*}
            \sum_{x\in D} a(x) q^{-\alpha\hi{x}}=0.
        \end{equation*}
    \end{enumerate}
\end{defin}

\begin{defin}
    Let $1<p\le\infty$, the atomic Hardy space $H^{1,p}_{\alpha}$ is the space of functions $f\in L^1_\alpha$ such that
    \begin{equation*}
        f=\sum_{i} \lambda_i a_i,
    \end{equation*}
    where the $a_i$ are $(1,p)$-atoms and $(\lambda_i)_i$ is a complex summable sequence. The $H^{1,p}_{\alpha}$ norm of $f$ is given by
    \begin{equation*}
        \|f\|_{H^{1,p}_{\alpha}}:=\inf\Bigl\{\sum_{i}|\lambda_i|:f=\sum_{i} \lambda_i a_i,\,a_i \text{ $(1,p)$-atoms}\Bigr\}.
    \end{equation*}
\end{defin}

Notice that by H\"older inequality a $(1,p_2)$-atom is also a $(1,p_1)$-atom for all $p_1\le p_2$. It easily follows that the inclusions
\begin{equation*}
    H^{1,\infty}_{\alpha}\subseteq H^{1,p_2}_{\alpha}\subseteq H^{1,p_1}_{\alpha}
\end{equation*}
hold for every $p_1\le p_2$.

\begin{defin}
    Let $f\colon\X\to\C$ and $D\in\mathcal{D}$, we indicate by $f_D$ the mean of $f$ on $D$ with respect to $\sigma_{\alpha}$, that is
    \begin{equation*}
        f_D=\frac{1}{\sigma_{\alpha}(D)}\sum_{x\in D} f(x)q^{-\alpha\hi{x}}.
    \end{equation*}
    Let $1\le r<\infty$, the bounded mean oscillation space $BMO^r_{\alpha}$ is the space of functions $f\colon\X\to\C$ such that
    \begin{equation*}
        \sup_{D\in\mathcal{D}}\frac{1}{\sigma_{\alpha}(D)}\sum_{x\in D} |f(x)-f_D|^r q^{-\alpha\hi{x}} <\infty,
    \end{equation*}
    modulo the constant functions. The $BMO^r_{\alpha}$-norm is given by
    \begin{equation*}
        \|f\|_{BMO^r_{\alpha}}:=\sup_{D\in\mathcal{D}}\Bigl(\frac{1}{\sigma_{\alpha}(D)}\sum_{x\in D} |f(x)-f_D|^r q^{-\alpha\hi{x}}\Bigr)^{\frac{1}{r}}.
    \end{equation*}
\end{defin}

Again by H\"older inequality we obtain the following inclusions
\begin{equation*}
    BMO^{r_2}_{\alpha}\subseteq BMO^{r_1}_{\alpha}\subseteq BMO^{1}_{\alpha}
\end{equation*}
for all $r_1\le r_2$.

In a doubling setting there are two outstanding results by R.R.~Coifman and G.~Weiss, presented in \cite{cw}, concerning $H^1$ and $BMO$ spaces. Theorem A gives that $H^{1,p}_{\alpha}=H^{1,\infty}_{\alpha}$ as vector spaces with equivalent norms for all $1<p\le\infty$. Furthermore by Theorem B, the space $BMO^{p'}_{\alpha}$ characterizes the dual of $H^{1,p}_{\alpha}$ for all $1<p\le\infty$ and $p'$ its conjugate. As a byproduct of these results we also obtain that $BMO^r_{\alpha}=BMO^1_{\alpha}$  as vector spaces with equivalent norms for all $1\le r<\infty$. Consequently we will denote the space $H^{1,\infty}_{\alpha}$ by $H^1_{\alpha}$ and the space $BMO^1_{\alpha}$ by $BMO_{\alpha}$.

\section{The orthonormal basis of $\B{2}$}

This section is devoted to give an explicit construction of an orthonormal basis for the space $\B{2}$.

For all $v\in\X$, consider the space
\begin{equation*}
    Z_v:=\Bigl\{\phi\colon S(v)\to\C : \sum_{x\in S(v)}\phi(x)=0\Bigr\}
    \simeq\C^{q-1}.
\end{equation*}
Let $\{a_{v,j}:j=1,\dots,q-1\}$ be an orthonormal basis for $Z_v$ and let $A_{v,j}$ be the extension by zero of $a_{v,j}$ to all of $\X$, namely $A_{v,j}:\X\to\C$ is given by
\begin{equation*}
    A_{v,j}(x)=
    \begin{dcases}
    a_{v,j}(x) &\mbox{if } x\in S(v), \\
    0 &\mbox{if } x\notin S(v).
    \end{dcases}
\end{equation*}
Observe that $A_{v,j}$ is harmonic and vanishes on the horoball $\HBo{\hi{v}}$. Now we set $g_{v,j}:=(A_{v,j})^H_{\hi{v}}$, that is
\begin{equation*}
    g_{v,j}(x)=
    \begin{dcases}
    0 &\mbox{if } x\notin \dot U_v, \\
    \Bigl(\sum_{i=0}^{\hi{x}-\hi{v}-1}q^{-i}\Bigr) A_{v,j}(p^{\hi{x}-\hi{v}-1}(x))&\mbox{if } x\in \dot U_v,
    \end{dcases}
\end{equation*}
since $A_{v,j}|_{\Ho{\hi{v}}}=0$.
It is clear that $g_{v,j}$ is harmonic on $\X$ for every $v\in\X$ and every $j=1,\dots,q-1$. In addition, by Example \ref{es} we know that it is bounded and that it belongs to $\B{2}$.
By Lemma \ref{sp}, the $\B{2}$ norm of $g_{v,j}$ is
\begin{equation}\label{ng}
    \|g_{v,j}\|_{\B{2}}^2=\langle g_{v,j},g_{v,j}\rangle_{\B{2}}
    =\sum_{y\in S(v)}\overline{g_{v,j}(y)}g_{v,j}(y)b_{\hi{v}}=b_{\hi{v}}\|a_{v,j}\|^2_{Z_v}=b_{\hi{v}}.
\end{equation}

\begin{thm}\label{cos}
    The family
    \begin{equation*}
        \mathcal{S}
        =\left\{g_{v,j}: v\in\X, \, j=1,\dots,q-1\right\}
    \end{equation*}
    is a complete orthogonal system for $\B{2}$ for every $\alpha>1$.
    \begin{proof}
        We start by showing that $\mathcal{S}$ is an orthogonal system in $\B{2}$, that is
        \begin{equation*}
            \langle\, g_{v,i},\,g_{w,j}\,\rangle_{\B{2}}=0 \quad \text{whenever } (v,i)\ne(w,j).
        \end{equation*}
        If $v\ne w$ then either $U_v\cap U_w=\varnothing$ or, without loss of generality, $U_v\cap U_w=U_v\ne U_w$.
        In the first case, $\supp(g_{v,i})\cap\supp(g_{w,j})=\varnothing$ for all $i,j\in\{1,\dots,q-1\}$, so that $\langle\, g_{v,i},\,g_{w,j}\,\rangle_{\B{2}}=0$.
        In the second case $U_v\subsetneq U_w$ implies that $\hi{v}>\hi{w}$, which gives us that $g_{v,i}(x)=0$ for all $x\in\HBo{\hi{w}+1}\subseteq\HBo{\hi{v}}$. By Lemma \ref{sp}, for every $i,j\in\{1,\dots,q-1\}$
        \begin{equation*}
            \langle\, g_{v,i},\,g_{w,j}\,\rangle_{\B{2}}
            =\sum_{y\in S(w)} \overline{g_{w,j}(y)}
            \Bigl(g_{v,i}(y)b_{\hi{w}}-g_{v,i}(w) b_{\hi{w}}'\Bigr)
            =0.
        \end{equation*}
        If $v=w$, let $i,j\in\{1,\dots,q-1\}$, $i\ne j$, we know that $g_{v,i}(x)=0$ for all $x\in\HBo{\hi{v}}$ and by Lemma \ref{sp}
        \begin{align*}
            \langle\, g_{v,i},\,g_{v,j}\,\rangle_{\B{2}}
            &=\sum_{y\in S(v)} \overline{g_{v,j}(y)}
            \Bigl(g_{v,i}(y)b_{\hi{v}}-g_{v,i}(v) b_{\hi{v}}'\Bigr)\\
            &=b_{\hi{v}}\sum_{y\in S(v)} \overline{a_{v,j}(y)}a_{v,i}(y)=0
        \end{align*}
        since $\{a_{v,j}:j=1,\dots,q-1\}$ is an orthogonal system in $Z_v$. We conclude that $\mathcal{S}$ is an orthogonal in $\B{2}$.
        
        In order to see that $\mathcal{S}$ is complete in $\B{2}$ we take $f\in\B{2}$ such that $\langle\,f,\,g_{v,j}\,\rangle_{\B{2}}=0$
        for all $v\in\X$ and $j\in\{1,\dots,q-1\}$.
        By Lemma \ref{sp}
        \begin{align*}
            \langle\, f,\,g_{v,j}\,\rangle_{\B{2}}
            &=\sum_{y\in S(v)} \overline{a_{v,j}(y)}
            \Bigl(f(y)b_{\hi{v}}-f(v) b_{\hi{v}}'\Bigr)\\
            &=b_{\hi{v}}\sum_{y\in S(v)} \overline{a_{v,j}(y)}f(y),
        \end{align*}
        where the last equality is due to the fact that $a_{v,j}\in Z_v$ and then it has vanishing mean on $S(v)$. Then for all $v\in\X$ and $j\in\{1,\dots,q-1\}$ we have
        \begin{equation*}
            \langle\, f|_{S(v)},\,a_{v,j}\,\rangle_{\C^{S(v)}}
            =\sum_{y\in S(v)} \overline{a_{v,j}(y)}f(y)
            =0.
        \end{equation*}
        Since $\{a_{v,j}:j=1,\dots,q-1\}$ is an orthonormal basis for $Z_v$ and $Z_v$ is a subspace of $\C^q$, $f|_{S(v)}$ must lie in the orthogonal of $Z_v$ in $\C^q$. It is clear that $Z_v^{\perp}$ has dimension 1 and it is easy to see that it is generated by $a_{v,q}$ given by
        \begin{equation*}
            a_{v,q}(y)=\frac{1}{\sqrt{q}}, \quad y\in S(v).
        \end{equation*}
        It follows that $f|_{S(v)}$ is constant, say $f|_{S(v)}\equiv c_v$ for every $v\in\X$. Now consider $x\in\X$ and the infinite chain $[x,\omega)=\{v_n\}_{n=0}^{+\infty}$. The function $f$ is harmonic in $v_{n+1}$ and constant on its successors for all $n\in\N$, hence
        \begin{equation*}
            f(v_{n+1})
            =\frac{1}{q+1}f(p(v_{n+1}))+\frac{1}{q+1}\sum_{y\in S(v_{n+1})} f(y)
            =\frac{1}{q+1}f(v_{n+2})+\frac{q}{q+1} f(v_{n}),
        \end{equation*}
        that is
        \begin{equation*}
            f(v_n)
            =\frac{q+1}{q}f(v_{n+1})-\frac{1}{q}f(v_{n+2}),
            \quad n\in\N.
        \end{equation*}
        Proceeding in a similar way as in the proof of Lemma \ref{p1} one can easily obtain that for every $k\in\N$
        \begin{equation*}
            f(v_0)
            =\Bigl(\sum_{i=0}^k q^{-i}\Bigr) f(v_{k})
            -\Bigl(\sum_{i=1}^k q^{-i}\Bigr) f(v_{k+1}).
        \end{equation*}
        In particular, since $f\in\B{2}$ implies that $f(v_k)\to0$ as $k\to+\infty$,
        \begin{equation*}
            f(v_0)
            =\lim_{k\to +\infty}\Bigl(\frac{q-q^{-k}}{q-1}f(v_{k})
            -\frac{1-q^{-k}}{q-1}f(v_{k+1})\Bigr)=0
        \end{equation*}
        This proves that $f\equiv0$ and then that $\mathcal{S}$ is complete in $\B{2}$.
    \end{proof}
\end{thm}

It is worth noticing that the complete orthogonal system $\mathcal{S}$ does not depend on $\alpha$. From Theorem \ref{cos} and \eqref{ng}, we immediately have the following result.

\begin{cor}
    For every $\alpha>1$ the family
    \begin{equation*}
        \mathcal{S}_{\alpha}
        =\Bigl\{\frac{g_{v,j}}{\sqrt{b_{\hi{v},\alpha}}}: v\in\X, \,j=1,\dots,q-1\Bigr\}
    \end{equation*}
    is an orthonormal basis for $\B{2}$.
\end{cor}

\section{The reproducing kernel for $\B{2}$}
The goal of this section is to show that $\B{2}$ is a reproducing kernel Hilbert space and to find an explicit formula for the kernel.

Fix $y\in\X$ and consider the evaluation functional $\Phi_y:\B{2}\to\C$ given by $\Phi_y f:=f(y)$. We have that
\begin{equation*}
    |\Phi_y(f)|^2
    =|f(y)|^2
    \le \frac{1}{\sigma_{\alpha}(y)} \sum_{x\in\X}|f(x)|^2\sigma_{\alpha}(x)
    =\frac{1}{\sigma_{\alpha}(y)}\|f\|^2_{2,\alpha}.
\end{equation*}
Thus $\Phi_y$ is continuous and $\B{2}$ is a reproducing kernel Hilbert space, that is for all $y\in\X$ there exists $ \mathcal K_{\alpha,y}\in\B{2}$ such that
\begin{equation*}
    f(y)=\langle f, \mathcal K_{\alpha,y}\rangle_{\B{2}}, \quad f\in\B{2}.
\end{equation*}
The operator $\mathcal{K}_{\alpha}:\X\times\X\to\C$ defined by $\mathcal{K}_{\alpha}(y,x):= \mathcal K_{\alpha,y}(x)$
is the \textit{reproducing kernel} of $\B{2}$.

In order to compute $ \mathcal K_{\alpha,y}$, we have to find the reproducing kernels of the spaces $Z_v$ as a preliminary step.
We fix $v\in\X$ and define the function
\begin{equation*}
    \Gamma_v(x,y)=
    \begin{dcases}
    0 &\mbox{if } \{x,y\}\notin \dot U_v, \\
    \frac{q-1}{q} &\mbox{if } \{x,y\}\in U_z \text{ for some }z\in S(v),\\
    -\frac{1}{q}&\mbox{otherwise}.
    \end{dcases}
\end{equation*}

\begin{rem}
    Let $v\in\X$. Some observations on $\Gamma_v$ are in order.
    \begin{enumerate}[1)]
        \item $\Gamma_v$ is symmetric, that is $\Gamma_v(x,y)=\Gamma_v(y,x)$ for all $x,y\in\X$;
        \item $\Gamma_v(x,\cdot)\equiv 0$ for all $x\notin \dot U_v$ and $\supp(\Gamma_v(x,\cdot))=\dot U_v$ for all $x\in \dot U_v$ and, in this case, the values of $\Gamma_v(x,\cdot)$ are completely determined by the values on $S(v)$. Indeed, if $x\in \dot U_v$ then $x\in U_z$ for some $z\in S(v)$ and
        \begin{equation*}
            \Gamma_v(x,y)=
    \begin{dcases}
    \frac{q-1}{q} &\mbox{if } y\in U_z,\\
    -\frac{1}{q} &\mbox{if } y\notin U_z.
    \end{dcases}
        \end{equation*}
A representation is given in Figure \ref{f2}.

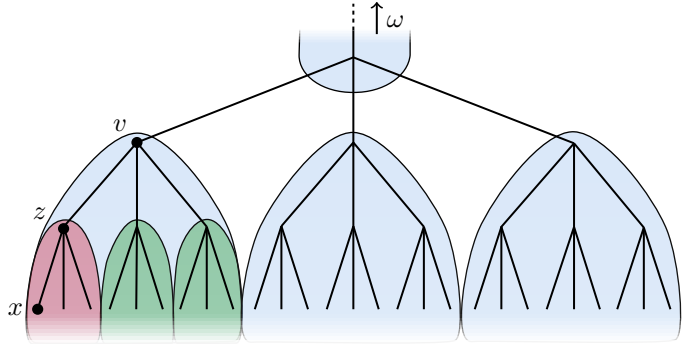
\begin{figure}[h]
\centering

  
\tikzset {_bjzmyzj0p/.code = {\pgfsetadditionalshadetransform{ \pgftransformshift{\pgfpoint{0 bp } { 0 bp }  }  \pgftransformrotate{-90 }  \pgftransformscale{2 }  }}}
\pgfdeclarehorizontalshading{_n4ulv0gr0}{150bp}{rgb(0bp)=(1,1,1);
rgb(38.57142857142857bp)=(1,1,1);
rgb(56.875bp)=(1,1,1);
rgb(59.91071428571429bp)=(1,1,1);
rgb(100bp)=(1,1,1)}
\tikzset{_lyp9mq92y/.code = {\pgfsetadditionalshadetransform{\pgftransformshift{\pgfpoint{0 bp } { 0 bp }  }  \pgftransformrotate{-90 }  \pgftransformscale{2 } }}}
\pgfdeclarehorizontalshading{_w27zwhn6s} {150bp} {color(0bp)=(transparent!100);
color(38.57142857142857bp)=(transparent!100);
color(56.875bp)=(transparent!90);
color(59.91071428571429bp)=(transparent!0);
color(100bp)=(transparent!0) } 
\pgfdeclarefading{_o8th4xx3n}{\tikz \fill[shading=_w27zwhn6s,_lyp9mq92y] (0,0) rectangle (50bp,50bp); } 

  
\tikzset {_dm2492vvk/.code = {\pgfsetadditionalshadetransform{ \pgftransformshift{\pgfpoint{0 bp } { 0 bp }  }  \pgftransformrotate{-90 }  \pgftransformscale{2 }  }}}
\pgfdeclarehorizontalshading{_oupfulhck}{150bp}{rgb(0bp)=(1,1,1);
rgb(37.5bp)=(1,1,1);
rgb(50.26785714285714bp)=(1,1,1);
rgb(54.19642857142857bp)=(1,1,1);
rgb(62.5bp)=(1,1,1);
rgb(100bp)=(1,1,1)}
\tikzset{_i676qejk3/.code = {\pgfsetadditionalshadetransform{\pgftransformshift{\pgfpoint{0 bp } { 0 bp }  }  \pgftransformrotate{-90 }  \pgftransformscale{2 } }}}
\pgfdeclarehorizontalshading{_rca11kdf5} {150bp} {color(0bp)=(transparent!0);
color(37.5bp)=(transparent!0);
color(50.26785714285714bp)=(transparent!0);
color(54.19642857142857bp)=(transparent!100);
color(62.5bp)=(transparent!100);
color(100bp)=(transparent!100) } 
\pgfdeclarefading{_27o1hv7oy}{\tikz \fill[shading=_rca11kdf5,_i676qejk3] (0,0) rectangle (50bp,50bp); } 
\tikzset{every picture/.style={line width=0.75pt}} 

\begin{tikzpicture}[x=0.75pt,y=0.75pt,scale=0.9,yscale=-1,xscale=1]

\draw  [color={rgb, 255:red, 0; green, 0; blue, 0 }  ,draw opacity=1, line width=0.5 ][fill={rgb, 255:red, 74; green, 144; blue, 226 }  ,fill opacity=0.2 ] (73.87,248.25) .. controls (84.08,248.82) and (176.33,248.33) .. (183.7,247.19) .. controls (191.06,246.05) and (188,204.05) .. (180.95,185.95) .. controls (173.9,167.86) and (148.39,131.62) .. (130.34,130.99) .. controls (112.29,130.36) and (85.21,168.15) .. (77.56,185.95) .. controls (69.91,203.76) and (63.67,247.67) .. (73.87,248.25) -- cycle ;

\draw  [color={rgb, 255:red, 0; green, 0; blue, 0 }  ,draw opacity=1, line width=0.5][fill={rgb, 255:red, 2; green, 138; blue, 10 }  ,fill opacity=0.35  ]
(112.4,248.82) .. controls (118.59,249.21) and (146.6,248.26) .. (148.8,248.63) .. controls (151,249) and (150,215) .. (147.17,202.54) .. controls (144.35,190.08) and (137.54,179.49) .. (129.64,179.51) .. controls (121.74,179.54) and (115.15,196.01) .. (113.09,202.15) .. controls (111.02,208.3) and (106.21,248.44) .. (112.4,248.82) -- cycle ;
\draw  [color={rgb, 255:red, 0; green, 0; blue, 0 }  ,draw opacity=1, line width=0.5 ][fill={rgb, 255:red, 2; green, 138; blue, 10 }  ,fill opacity=0.35 ]
(152.26,248.54) .. controls (156.74,248.64) and (179.16,248.34) .. (184.66,248.34) .. controls (190.17,248.34) and (187,211) .. (185.04,202.25) .. controls (183.08,193.5) and (176.41,179.21) .. (168.51,179.23) .. controls (160.61,179.25) and (153.57,192.07) .. (151.95,201.87) .. controls (150.33,211.67) and (147.79,248.43) .. (152.26,248.54) -- cycle ;
\draw  [color={rgb, 255:red, 0; green, 0; blue, 0 }  ,draw opacity=1, line width=0.5 ][fill={rgb, 255:red, 208; green, 2; blue, 27 }  ,fill opacity=0.35 ] (71.81,248.82) .. controls (78,249.21) and (100.7,248.63) .. (106.21,248.63) .. controls (111.71,248.63) and (110,216.33) .. (106.58,202.54) .. controls (103.17,188.74) and (97.95,179.49) .. (90.05,179.51) .. controls (82.15,179.54) and (76.66,187.31) .. (72.5,202.15) .. controls (68.33,217) and (65.62,248.44) .. (71.81,248.82) -- cycle ;
\draw  [color={rgb, 255:red, 0; green, 0; blue, 0 }  ,draw opacity=1, line width=0.5 ][fill={rgb, 255:red, 74; green, 144; blue, 226 }  ,fill opacity=0.2  ]
(193.24,247.86) .. controls (204.52,247.07) and (300.42,249.96) .. (306.06,246.81) .. controls (311.7,243.66) and (308.37,203.66) .. (301.32,185.57) .. controls (294.27,167.47) and (267.76,131.23) .. (249.71,130.6) .. controls (235.09,130.09) and (214.56,154.77) .. (203.21,173.63) .. controls (200.55,178.07) and (198.39,182.18) .. (196.93,185.57) .. controls (189.28,203.37) and (181.96,248.65) .. (193.24,247.86) -- cycle ;
\draw  [color={rgb, 255:red, 0; green, 0; blue, 0 }  ,draw opacity=1, line width=0.5][fill={rgb, 255:red, 74; green, 144; blue, 226 }  ,fill opacity=0.2  ]
(315.36,247.48) .. controls (326.64,246.69) and (422.54,249.57) .. (428.18,246.42) .. controls (433.82,243.27) and (430.49,203.28) .. (423.44,185.18) .. controls (416.39,167.09) and (389.88,130.85) .. (371.83,130.22) .. controls (353.78,129.59) and (326.7,167.38) .. (319.05,185.18) .. controls (311.4,202.99) and (304.08,248.26) .. (315.36,247.48) -- cycle ;
\draw  [draw opacity=0][shading=_n4ulv0gr0,_bjzmyzj0p,path fading= _o8th4xx3n ,fading transform={xshift=2}] (53.79,127.86) -- (437.13,127.86) -- (437.13,259) -- (53.79,259) -- cycle ;
\draw  [line width=0.5,
][fill={rgb, 255:red, 74; green, 144; blue, 226 }  ,fill opacity=0.2  ] (281.08,49.6) -- (281.08,87.38) .. controls (279.98,99.08) and (266.32,108.42) .. (249.61,108.23) .. controls (232.87,108.42) and (219.41,99.08) .. (219.53,87.38) -- (220.51,49.6) -- cycle ;
\draw  [draw opacity=0][shading=_oupfulhck,_dm2492vvk,path fading= _27o1hv7oy ,fading transform={xshift=2}]
(216.53,45.6) -- (284.08,45.6) -- (284.08,96.6) -- (216.53,96.6) -- cycle ;

\draw  (129.65,136.21) -- (249.61,88.63) ;
\draw    (371.95,136.66) -- (249.61,88.63) ;
\draw    (371.95,183.5) -- (371.95,136.66) ;
\draw    (410.64,183.74) -- (371.95,136.66) ;
\draw    (331.26,183.82) -- (371.95,136.66) ;
\draw    (410.64,229.42) -- (410.64,183.74) ;
\draw    (425.4,229.42) -- (410.64,183.74) ;
\draw    (395.88,229.42) -- (410.64,183.74) ;
\draw    (371.95,229.42) -- (371.95,183.42) ;
\draw    (386.31,229.42) -- (371.95,183.42) ;
\draw    (356.79,229.42) -- (371.95,183.42) ;
\draw    (331.66,229.42) -- (331.66,183.42) ;
\draw    (346.42,229.42) -- (331.66,183.42) ;
\draw    (316.9,229.42) -- (331.66,183.42) ;
\draw    (249.61,182.97) -- (249.61,136.21) ;
\draw    (288.7,183.29) -- (249.61,136.21) ;
\draw    (209.72,182.97) -- (249.61,136.21) ;
\draw    (288.7,229.42) -- (288.7,183.29) ;
\draw    (303.46,229.42) -- (288.7,183.29) ;
\draw    (273.94,229.42) -- (288.7,183.29) ;
\draw    (249.61,229.42) -- (249.61,182.97) ;
\draw    (263.68,229.42) -- (249.61,182.97) ;
\draw    (234.85,229.42) -- (249.61,182.97) ;
\draw    (209.72,229.42) -- (209.72,182.97) ;
\draw    (224.48,229.42) -- (209.72,182.97) ;
\draw    (194.96,229.42) -- (209.72,182.97) ;
\draw    (129.65,183.05) -- (129.65,136.21) ;
\draw    (168.34,183.29) -- (129.65,136.21) ;
\draw    (88.97,183.37) -- (129.65,136.21) ;
\draw    (168.34,229.42) -- (168.34,183.29) ;
\draw    (183.1,229.42) -- (168.34,183.29) ;
\draw    (153.58,229.42) -- (168.34,183.29) ;
\draw    (129.25,229.42) -- (129.65,183.05) ;
\draw    (144.01,229.42) -- (129.65,183.05) ;
\draw    (114.49,229.42) -- (129.65,183.05) ;
\draw    (88.97,229.42) -- (88.97,183.37) ;
\draw    (104.12,229.42) -- (88.97,183.37) ;
\draw    (74.61,229.42) -- (88.97,183.37) ;
\draw    (249.61,138.63) -- (249.61,73.63) ;
\draw [dashed, dash pattern={on 1.5 pt off 1.5pt}]   (249.61,73.63)--(249.61,58.63) ;
\draw [->]   (262.61,75.63)--(262.61,57.63) ;

\filldraw  [color=black,line width=0.2mm ,draw opacity=1 ][fill=black  ,fill opacity=1 ] (74.61,229.42) circle (2pt);
\filldraw  [color=black,line width=0.2mm ,draw opacity=1 ][fill=black  ,fill opacity=1 ] (129.65,136.21) circle (2pt);
\filldraw  [color=black,line width=0.2mm ,draw opacity=1 ][fill=black  ,fill opacity=1 ] (88.97,184.37) circle (2pt);

	\node at (72,229.42) [anchor= east] {$x$};
	\node at (262.61,68.63) [anchor=west] {$\omega$};
	\node at (129.65,135.21) [anchor=south east] {$v$};
    \node at (84.97,184.37) [anchor=south east] {$z$};

\end{tikzpicture}

    \caption{Realization of $\Gamma_v(x,\cdot)$ for $x\in U_z$, $z\in S(v)$. The function vanishes on the blue area, while its value is $\frac{q-1}{q}$ on the red sector and $-\frac{1}{q}$ on the green ones
    \label{f2}}

\end{figure}

        \item The function $\Gamma_v$ is the reproducing kernel of $Z_v$, that is for every $\varphi\in Z_v$ and $x\in S(v)$,
        \begin{equation*}
        \varphi(x)
        =\langle\,\varphi,\,\Gamma_v(x,\cdot)|_{S(v)}\,\rangle_{Z_v}.
        \end{equation*}
        Indeed,
        \begin{equation*}
            \sum_{y\in S(v)} \Gamma_v(x,y)
            =\frac{q-1}{q}-\frac{1}{q}(q-1)
            =0
        \end{equation*}
        and then $\Gamma_v(x,\cdot)|_{S(v)}\in Z_v$. Now, if $\varphi\in Z_v$, then 
        \begin{align*}
            \langle\,\varphi,\,\Gamma_v(x,\cdot)|_{S(v)}\,\rangle_{Z_v}
            &=\sum_{y\in S(v)}\varphi(y)\overline{\Gamma_v(x,y)}\\
            &=\frac{q-1}{q}\varphi(x)-\frac{1}{q}\sum_{y\in S(v)\setminus\{x\}}\varphi(y)\\
            &=\frac{q-1}{q}\varphi(x)-\frac{1}{q}(-\varphi(x))
            =\varphi(x),
        \end{align*}
        where we used the fact that $\varphi$ has vanishing mean on $S(v)$.
        \item The function $\Gamma_v(x,\cdot)$ is harmonic on $\HBo{\hi{v}}$, indeed it vanishes there and it is sufficient to check the harmonicity at $v$ in the case $x\in \dot U_v$, which follows by $\supp(\Gamma_v(x,\cdot))=\dot U_v$ and $\Gamma_v(x,\cdot)|_{S(v)}\in Z_v$.
        This allows us to consider, for all $v\in\X$, the harmonic extension $(\Gamma_{p(v)}(v,\cdot))^H_{\hi{p(v)}}$, that is
        \begin{equation*}
        (\Gamma_{p(v)}(v,\cdot))^H_{\hi{p(v)}}=
        \begin{dcases}
            0 &\mbox{if } y\notin \dot U_{p(v)},\\
            \Bigl(\sum_{i=0}^{\hi{y}-\hi{v}}q^{-i}\Bigr)\frac{q-1}{q} &\mbox{if } y\in U_v,\\
            \Bigl(\sum_{i=0}^{\hi{y}-\hi{v}}q^{-i}\Bigr)\Bigl(-\frac{1}{q}\Bigr) &\mbox{if } y\in \dot U_{p(v)}\setminus U_v.
        \end{dcases}
        \end{equation*}
        It is clear that $|(\Gamma_{p(v)}(v,\cdot))^H_{\hi{p(v)}}(y)|\le1$ for all $y\in \dot U_v$ and this easily implies that $(\Gamma_{p(v)}(v,\cdot))^H_{\hi{p(v)}}$ is in $\B{2}$.
    \end{enumerate}
\end{rem}

From now on, for $v\in\X$ and $n\in\N$, we will denote
\begin{align*}
    &\Gamma_{n,v}=\Gamma_{p^{n+1}(v)} (p^n(v),\cdot),\\
    &\Gamma_{n,v}^H=\Bigl(\Gamma_{p^{n+1}(v)} (p^n(v),\cdot)\Bigr)^H_{\hi{p^{n+1}(v)}}.
\end{align*}

\begin{lemma}\label{sp2}
    Let $v\in\X$ and $f\in\B{2}$. For every $n\in\N$ the following holds
    \begin{align*}
    \langle\, f,\,\Gamma_{n,v}^H \,\rangle_{\B{2}}
    =b_{\hi{p^{n+1}(v)}} \Bigl(f(p^n(v))-\frac{q+1}{q} f(p^{n+1}(v)) +\frac{1}{q} f(p^{n+2}(v))\Bigr).
    \end{align*}
    \begin{proof}
        Let $v\in\X$ and $f\in\B{2}$, by Lemma \ref{sp}
        \begin{align*}
            \langle\, f,\,\Gamma_{0,v}^H \,\rangle_{\B{2}}
            &=\sum_{y\in S(p(v))}\overline{\Gamma_{0,v}(y)}
            \Bigl(f(y) b_{\hi{p(v)}} -f(p(v)) b_{\hi{p(v)}}'\Bigr)\\
            &=\sum_{y\in S(p(v))}\overline{\Gamma_{0,v}(y)}
            f(y) b_{\hi{p(v)}} -f(p(v)) b_{\hi{p(v)}}'\sum_{y\in S(p(v))}\overline{\Gamma_{0,v}(y)}.
        \end{align*}
        We remind that $\Gamma_{0,v}\in Z_{p(v)}$ and so the second term vanishes. Using that $f$ is harmonic in $p(v)$, we obtain
        \begin{align*}
            \langle\, f,\,\Gamma_{0,v}^H \,\rangle_{\B{2}}
            &=\sum_{y\in S(p(v))}\overline{\Gamma_{0,v}(y)}
            f(y) b_{\hi{p(v)}}\\
            &=b_{\hi{p(v)}}\Bigl(\frac{q-1}{q}f(v)-\frac{1}{q}\sum_{y\in S(p(v))\setminus\{v\}}f(y)\Bigr)\\
            &=b_{\hi{p(v)}}\Bigl(\frac{q-1}{q}f(v)-\frac{q+1}{q}f(p(v))+\frac{1}{q}f(v)+\frac{1}{q}f(p^2(v))\Bigr)\\
            &=b_{\hi{p(v)}}\Bigl(f(v)-\frac{q+1}{q}f(p(v))+\frac{1}{q}f(p^2(v))\Bigr).
        \end{align*}
        Upon replacing $v$ with $p^n(v)$ we obtain the claim for every $n\in\N$.
    \end{proof}
\end{lemma}

\begin{thm}\label{rkv}
    For every $v\in\X$ the function
    \begin{equation*}
         \mathcal K_{\alpha,v}=\sum_{n=0}^{+\infty} d_{n,\hi{v}}^{\alpha}\Gamma_{n,v}^H
    \end{equation*}
    is the kernel associated to $v$, where
    \begin{equation*}
        d_{n,k}^{\alpha}
        :=\Bigl(\sum_{j=0}^n q^{-j}\Bigr)b_{k-n-1}^{-1}.
    \end{equation*}
\end{thm}

Before giving the proof of this Theorem we show that $\mathcal K_{\alpha,v}$ belongs to $\B{2}$.

\begin{lemma}\label{lemma.kib}
    $\mathcal K_{\alpha,v}\in\B{2}$ for every $v\in\X$.
\end{lemma}
    \begin{proof}
        Define $d_{n,k}^{\alpha}$ as above and consider for every $m\in\N$ the partial sums given by
        \begin{equation*}
            s_v^m=\sum_{n=0}^{m}d_{n,\hi{v}}^{\alpha}\Gamma_{n,v}^H.
        \end{equation*}
        It is clear that $s_v^m\in\B{2}$ for every $m\in\N$ because it is a finite sum of functions in $\B{2}$. Our goal is to show that $\|\mathcal K_{\alpha,v}-s_v^m\|_{2,\alpha}\to0$. Let $x\in U_{p^{m+1}(v)}$,
        \begin{align*}
            |(\mathcal K_{\alpha,v}-s_{v}^m)(x)|
            &=\Bigl|\sum_{n=m+1}^{+\infty}d_{n,\hi{v}}^{\alpha}
            \Gamma_{n,v}^H(x)\Bigr|
            \le \sum_{n=m+1}^{+\infty} d_{n,\hi{v}}^{\alpha}
            \Bigl|\Gamma_{n,v}^H(x)\Bigr|\\
            &\le\sum_{n=m+1}^{+\infty} d_{n,\hi{v}}^{\alpha}
            \lesssim\sum_{n=m+1}^{+\infty}q^{\alpha(\hi{v}-n-1)}\\
            &\simeq q^{\alpha\hi{v}} \sum_{n=m+1}^{+\infty}q^{-\alpha n}
            \lesssim q^{\alpha(\hi{v}-m)}.
        \end{align*}
        Now let $x\notin U_{p^{m+1}(v)}$, in this case the first sector of the form $U_{p^n(v)}$ which contains $x$ is $U_{p^{\hi{v}-\hi{x\wedge v}}(v)}=U_{x\wedge v}$. It follows that $\Gamma_{n,v}^H(x)=0$ for every $n=m+1,\dots,\hi{v}-\hi{x\wedge v}+1$.
        As in the previous computation we obtain
        \begin{equation*}
            |(\mathcal K_{\alpha,v}-s_v^m)(x)|
            =\Bigl|\sum_{n=\hi{v}-\hi{x\wedge v}}^{+\infty}d_{n,\hi{v}}^\alpha
            \Gamma_{n,v}^H(x)\Bigr|
            \lesssim q^{\alpha \hi{x\wedge v}}.
        \end{equation*}
        Then we get
        \begin{align*}
            &\|\mathcal K_{\alpha,v}-s_v^m\|^2_{2,\alpha}
            =\sum_{x\in\X}|(\mathcal K_{\alpha,v}-s_v^m)(x)|^2\sigma_{\alpha}(x)\\
            &=\hspace{-0.2cm}\sum_{x\in U_{p^{m+1}(v)}}\hspace{-0.2cm}|(\mathcal K_{\alpha,v}-s_v^m)(x)|^2\sigma_{\alpha}(x)
            +\hspace{-0.2cm}\sum_{x\notin U_{p^{m+1}(v)}}\hspace{-0.1cm}|(\mathcal K_{\alpha,v}-s_v^m)(x)|^2\sigma_{\alpha}(x)\\
            &=\hspace{-0.2cm}\sum_{x\in U_{p^{m+1}(v)}}\hspace{-0.2cm}|(\mathcal K_{\alpha,v}-s_v^m)(x)|^2\sigma_{\alpha}(x)
            +\hspace{-0.2cm}\sum_{i=\hi{p^{m+2}(v)}}^{-\infty}\sum_{\hi{x\wedge v}=i}\hspace{-0.2cm}|(\mathcal K_{\alpha,v}-s_v^m)(x)|^2\sigma_{\alpha}(x)\\
            &\lesssim\sum_{x\in U_{p^{m+1}(v)}}q^{2\alpha(\hi{v}-m)}\sigma_{\alpha}(x)
            +\sum_{i=\hi{v}-m-2}^{-\infty}\sum_{\hi{x\wedge v}=i}q^{2\alpha\hi{x\wedge v}}\sigma_{\alpha}(x)\\
            &\lesssim q^{2\alpha(\hi{v}-m)}\sigma_{\alpha}(U_{p^{m+1}(v)})+\sum_{i=m+2-\hi{v}}^{+\infty} q^{-2\alpha i} \sigma_{\alpha}(\{x:\hi{x\wedge v}=-i\}).
        \end{align*}
        Notice that $\{x:\hi{x\wedge v}=-i\}=U_{p^{\hi{v}+i}(v)}\setminus U_{p^{\hi{v}+i-1}(v)}$, so that by \eqref{ms}
        \begin{align}\label{dm}
            \notag\sigma_{\alpha}(\{x:\hi{x\wedge v}=-i\})
            &=\frac{1}{1-q^{1-\alpha}}\left(q^{-\alpha\hi{p^{\hi{v}+i}(v)}}-q^{-\alpha\hi{p^{\hi{v}+i-1}(v)}}\right)\\
            \notag&=\frac{1}{1-q^{1-\alpha}}\left(q^{\alpha i}-q^{\alpha(i-1)}\right)
            =q^{\alpha i}\frac{1-q^{-\alpha}}{1-q^{1-\alpha}}
            \simeq q^{\alpha i}.
        \end{align}
        Therefore we obtain
        \begin{align*}
            \|\mathcal K_{\alpha,v}-s_v^m\|^2_{2,\alpha}&\lesssim q^{2\alpha(\hi{v}-m)}q^{-\alpha\hi{p^{m+1}(v)}}
            +\sum_{i=m+2-\hi{v}}^{+\infty} q^{-2\alpha i}q^{\alpha i}\\
            &\lesssim q^{\alpha(\hi{v}-m)}
            +\sum_{i=m+2-\hi{v}}^{+\infty} q^{-\alpha i}\xrightarrow[m \rightarrow +\infty]{}0.
        \end{align*}
        This proves that $\mathcal K_{\alpha,v}\in\B{2}$ since it is closed in $\Lda$.
    \end{proof}

    We are in a position to prove the main theorem.
    \begin{proof}[Proof of Theorem~\ref{rkv}]
        Let $f\in\B{2}$, we claim that
        \begin{equation*}
            f(v)=\langle \, f\,,\,\mathcal K_{\alpha,v} \rangle_{\B{2}}.
        \end{equation*}
        By Lemma \ref{sp2} we have
        \begin{align*}
            \langle \, f\,,\sum_{n=0}^{+\infty}d_{n,\hi{v}}^\alpha
            &\Gamma_{n,v}^H \rangle_{\B{2}}
            =\sum_{n=0}^{+\infty}d_{n,\hi{v}}^\alpha
            \langle \, f\,,\,
            \Gamma_{n,v}^H\rangle_{\B{2}}\\
            =&\sum_{n=0}^{+\infty}\Bigl(\sum_{j=0}^n q^{-j}\Bigr) \Bigl(f(p^n(v))-\frac{q+1}{q} f(p^{n+1}(v)) +\frac{1}{q} f(p^{n+2}(v))\Bigr)\\
            =&f(v)-\frac{q+1}{q}f(p(v))+\Bigl(1+\frac{1}{q}\Bigr)f(p(v))\\
            &+\sum_{m=2}^{+\infty} f(p^m(v))\Bigl(\sum_{j=0}^{m}q^{-j}-\frac{q+1}{q}\sum_{j=0}^{m-1}q^{-j}+\frac{1}{q}\sum_{j=0}^{m-2}q^{-j}\Bigr)\\
            =&f(v).
        \end{align*}
        Therefore $\mathcal K_{\alpha,v}$ has the reproducing property and it is the kernel associated to $v$.
    \end{proof}

Now we give an explicit formula for the kernel $\mathcal{K}_{\alpha}$ which exhibits its symmetry.

\begin{thm}\label{rkv2}
    The kernel $\mathcal{K}_{\alpha}(v,\cdot)=\mathcal K_{\alpha,v}$ may be rewritten as
    \begin{equation*}
        \mathcal{K}_{\alpha}(v,x)
        =\sum_{m=-\hi{x\wedge v}-1}^{+\infty} b_{-m-1}^{-1}
        \Bigl(\sum_{j=0}^{m+\hi{v}}q^{-j}\Bigr)\Bigl(\sum_{i=0}^{m+\hi{x}}q^{-i}\Bigr)
        \Gamma_{m+\hi{v},v}(p^{m+\hi{x}}(x)).
    \end{equation*}
    In particular, $\mathcal{K}_\alpha$ is real and symmetric.
\end{thm}
    \begin{proof}
        By Theorem \ref{rkv}
         \begin{equation*}
            \mathcal{K}_{\alpha}(v,x)=\mathcal K_{\alpha,v}(x)=\sum_{n=0}^{+\infty} d_{n,\hi{v}}^{\alpha}
            \Gamma_{n,v}^H(x).
        \end{equation*}
        Recall that $\Gamma_{n,v}^H(x)\ne 0$ if and only if $x\in \dot U_{p^{n+1}(v)}$. The smaller $n\in\N$ such that $x\in \dot U_{p^{n+1}(v)}$ is given by $p^{n+1}(v)=x\wedge v$. It follows that $\hi{v}-n-1=\hi{x\wedge v}$, thus $n=\hi{v}-\hi{x\wedge v}-1$. We obtain
        \begin{align*}
            \mathcal{K}_{\alpha}(v,x)
            &=\sum_{n=\hi{v}-\hi{x\wedge v}-1}^{+\infty} d_{n,\hi{v}}^{\alpha}
            \Gamma_{n,v}^H(x)\\
            &=\sum_{n=\hi{v}-\hi{x\wedge v}-1}^{+\infty} d_{n,\hi{v}}^{\alpha}
            \Bigl(\sum_{i=0}^{\hi{x}-\hi{p^{n+1}(v)}-1} q^{-i}\Bigr)
            \Gamma_{n,v}(p^{\hi{x}-\hi{p^{n+1}(v)}-1}(x))\\
            &=\sum_{n=\hi{v}-\hi{x\wedge v}-1}^{+\infty} d_{n,\hi{v}}^{\alpha}
            \Bigl(\sum_{i=0}^{\hi{x}-\hi{v}+n} q^{-i}\Bigr)
            \Gamma_{n,v}(p^{\hi{x}-\hi{v}+n}(x)).
        \end{align*}
        Now we put $m=n-\hi{v}$ and notice that
        \begin{equation*}
            d_{m+\hi{v},\hi{v}}^{\alpha}=\Bigl(\sum_{j=0}^{m+\hi{v}}q^{-j}\Bigr) b_{-m-1}^{-1}.
        \end{equation*}
        In particular we have that
        \begin{equation*}
            p^{n+1}(v)=p^{m+1+\hi{x\wedge v}}(x\wedge v)
        \end{equation*}
        since $\hi{p^{n+1}(v)}=\hi{v}-n-1=-m-1=\hi{p^{m+1+\hi{x\wedge v}}(x\wedge v)}$ and both $p^{n+1}(v)$ and $p^{m+1+\hi{x\wedge v}}(x\wedge v)$ belong to $[v,\omega)$. It follows that
        \begin{align*}
            \mathcal{K}_{\alpha}(v,x)
            =\sum_{m=-\hi{x\wedge v}-1}^{+\infty} & b_{-m-1}^{-1} \Bigl(\sum_{j=0}^{m+\hi{v}}q^{-j}\Bigr)
            \Bigl(\sum_{i=0}^{m+\hi{x}} q^{-i}\Bigr)\\
            &\Gamma_{p^{m+1+\hi{x\wedge v}}(x\wedge v)}(p^{m+\hi{v}}(v),p^{m+\hi{x}}(x)).
        \end{align*}
        It is clear that $p^{m+1+\hi{x\wedge v}}(x\wedge v)=p^{m+1+\hi{v}}(v)$ so that
        \begin{equation*}
        \Gamma_{m+\hi{v},v}(p^{m+\hi{x}}(x))=
        \Gamma_{p^{m+1+\hi{x\wedge v}}(x\wedge v)}(p^{m+\hi{v}}(v),p^{m+\hi{x}}(x))
        \end{equation*}
        and the rest of the statement follows.
    \end{proof}

\begin{rem}
    A standard way to find the reproducing kernel $\mathcal K$ for a RKHS when an orthonormal basis $\mathcal B$ is already known exploits the following weak equality
    \begin{equation*}
        \mathcal{K}_{v}=\sum_{g\in\mathcal{B}} \langle \mathcal{K}_{v},g\rangle g=\sum_{g\in\mathcal{B}} g(v)g.
    \end{equation*}
    In our case this gives
    \begin{equation}\label{eqoss}
        \mathcal{K}_{\alpha,v}=\sum_{g\in\mathcal{S}_\alpha} g(v)g
        =\sum_{n=0}^{+\infty}b_{|v|-n}^{-1}\sum_{j=1}^{q-1}g_{p^n(v),j}(v)g_{p^n(v),j}.
    \end{equation}
    It is possible to see that the right hand side of \eqref{eqoss} converges in $L^2$ to the kernel as written in Theorem~\ref{rkv2}. We omit the explicit computation which turns out to be cumbersome.
\end{rem}

\section{Boundedness of integral operators}\label{sec.br}

The main goal of this last section is to provide boundedness results for integral operators, with a particular focus on the boundedness properties of the horocyclic Bergman projection.

By Proposition \ref{hs} the horocyclic harmonic Bergman space $\B{2}$ is closed in $L^2_\alpha$. Hence it is natural to consider the orthogonal projection $P_{\alpha}$ of $L^2_\alpha$ onto $\B{2}$, that is clearly given by
\begin{equation*}
    P_{\alpha}f=\sum_{x\in\X} \mathcal{K}_{\alpha}(\cdot,x)f(x)\sigma_{\alpha}(x), \quad f\in L^2_\alpha.
\end{equation*}
We refer to it as the \textit{horocyclic Bergman projection}. The aim is to prove that the extension of $P_{\alpha}$ to $L^p_\alpha$ is bounded for all $p>1$ and is of weak type $(1,1)$. Furthermore, we obtain boundedness limiting results involving $H^1_{\alpha}$ and $BMO_{\alpha}$.

Recall that by Proposition \ref{gmd} the measure metric space $(\X,\rho,\sigma_{\alpha})$ is doubling. Thus, by classical results~\cite{m}, a Calder\'on-Zygmund decomposition holds for functions in $L^1_\alpha$ and is based on the sets of the family $\mathcal D$ introduced in Section~\ref{subs.hbmo}.
\begin{prop}
    Let $f\in L^1(\X,\sigma_\alpha)$ and $\lambda>0$. There exist a family of disjoint sets $D_j\in\mathcal{D}$ and functions $g,b_j$ such that
    \begin{enumerate}[i)]
        \item $|f(x)|\le \lambda$ for all $\displaystyle x\in\X\setminus \bigsqcup_{j} D_j$;
        \item $\displaystyle f=g+\sum_{j}b_j$;
        \item $\|g\|_{2,\alpha}^2\lesssim \lambda\|f\|_{1,\alpha}$;
        \item every function $b_j$ is supported in $D_j$ and is such that
        \begin{equation*}
            \sum_{x\in D_j}b_j(x)q^{-\alpha\hi{x}}=0 \qquad\text{ and } \qquad\sum_{j}\|b_j\|_{1,\alpha}\lesssim \|f\|_{1,\alpha}.
        \end{equation*}
    \end{enumerate}
\end{prop}

In a doubling metric space $(X,d,\mu)$ the standard integral H\"ormander's condition for a kernel $K:X\times X\to \C$ is given by
\begin{equation*}
    \sup_{v\in\X,r>0}\,\sup_{x,y\in B_d(v,r)}\int_{X\setminus B_d(v,2r)}|K(z,x)-K(z,y)|\mu(z)<+\infty.
\end{equation*}
In our setting we have $X=\X$, $d=\rho$ and $\mu=\sigma_{\alpha}$ for some $\alpha>1$ and $(\X,\rho,\sigma_{\alpha})$ is doubling. Recall that by Remark \ref{gb} the Gromov balls are
\begin{equation*}
    B_{\rho}(x,r)=
    \begin{dcases}
    \{x\} &\mbox{if } \lfloor -\log(r)\rfloor>\hi{x},\\
    U_{p^{\hi{x}-\lfloor -\log(r)\rfloor}(x)} &\mbox{if } \lfloor -\log(r)\rfloor\le\hi{x}.
    \end{dcases}
\end{equation*}
It easily follows that the H\"ormander's condition for a kernel $\mathcal{K}\colon \X\times\X\to\C$ reads
\begin{equation}\label{Hc}
    \sup_{v\in\X}\,\sup_{x,y\in U_v} \sum_{z\in\X\setminus U_v} |\mathcal{K}(z,x)-\mathcal{K}(z,y)| \sigma_{\alpha}(z)<+\infty.
\end{equation}

The following theorem gathers some classical results concerning boundedness of integral operators on spaces of homogeneous type.  These results are mainly based on the doubling property and on the Calder\'on-Zygmund decomposition. The proof of $i)$ and $iii)$ is an easy adaptation to our setting of the proof of Theorem 3,~Ch.I~\cite{s} (see also~\cite{h}), while $ii)$ and $iv)$ follow from Theorem 8.2,~\cite{cmm}.

\begin{thm}\label{thmbound}
    Let $\mathcal T$ be an integral operator with kernel $\mathcal K\colon \X\times\X\to\C$. If $\mathcal T$ is bounded on $L^2_\alpha$ and $\mathcal K$ satisfies~\eqref{Hc}, then:
    \begin{enumerate}[i)]
        \item $\mathcal{T}$ is of weak type $(1,1)$ and by interpolation it is bounded on $L^p_{\alpha}$ for all $1<p<2$;
        \item $\mathcal{T}$ is bounded from $H^1_\alpha$ to $L^1_\alpha$.
    \end{enumerate}
    Furthermore, if the dual kernel $\mathcal K^* (x,y)=\overline{\mathcal K (y,x)}$ satisfies~\eqref{Hc}, then:
    \begin{enumerate}
        \item[iii)] $\mathcal{T}$ is bounded on $L^p_{\alpha}$ for all $2<p<\infty$;
        \item[iv)] $\mathcal{T}$ is bounded from $L^\infty_\alpha$ to $BMO_\alpha$.
    \end{enumerate}
\end{thm}

Now we show that the Bergman kernel $\mathcal K_\alpha$ satisfies the assumptions of the previous theorem.

\begin{prop}\label{prophorm}
    The Bergman kernel $\mathcal{K}_{\alpha}$ of $\B{2}$ satisfies integral H\"ormander's condition \eqref{Hc} for $\sigma_{\alpha}$.
    \begin{proof}
        We start by proving \eqref{Hc} with $y=v$. In this case, for every $x\in U_v$ and $z\notin U_v$ we have that $x\wedge z=v\wedge z$. Furthermore,
        \begin{equation}\label{eq.predug}
            p^{m+\hi{x}}(x)=p^{m+\hi{v}}(v)
        \end{equation}
        for every $m\ge-\hi{v}$ since they both have horocyclic index $-m$ and lie on $[v,\omega)$. It is clear that $\hi{v}\ge\hi{v\wedge z}+1$ because $z\notin U_v$, thus \eqref{eq.predug} holds true in particular for all $m\ge-\hi{v\wedge z}-1$.
        Hence for all $m\ge-\hi{v\wedge z}-1$
        \begin{equation*}
            \Gamma_{m+\hi{z},z}(p^{m+\hi{x}}(x))
            =\Gamma_{m+\hi{z},z}(p^{m+\hi{v}}(v))
        \end{equation*}
        and by Theorem \ref{rkv2}
        \begin{align*}
            |\mathcal{K}_{\alpha}&(z,x)-\mathcal{K}_{\alpha}(z,v)|\\
            &\le \sum_{m=-\hi{v\wedge z}-1}^{+\infty} b_{-m-1}^{-1}
            \Bigl( \sum_{j=0}^{m+\hi{z}} q^{-j} \Bigr)
            \Bigl( \sum_{i=m+1+\hi{v}}^{m+\hi{x}} q^{-i} \Bigr)
            \left|\Gamma_{m+\hi{z},z}(p^{m+\hi{v}}(v))\right|\\
            &\le \sum_{m=-\hi{v\wedge z}-1}^{+\infty} b_{-m-1}^{-1}
            \Bigl( \sum_{j=0}^{+\infty} q^{-j} \Bigr)
            \Bigl( \sum_{i=m+1+\hi{v}}^{+\infty} q^{-i} \Bigr)\\
            &\lesssim \sum_{m=-\hi{v\wedge z}-1}^{+\infty} q^{-\alpha m} q^{-m-\hi{v}}.
        \end{align*}
        Then we have
        \begin{align*}
            \sum_{z\in\X\setminus U_v}|\mathcal{K}_{\alpha}(z,x)&-\mathcal{K}_{\alpha}(z,v)|\sigma_{\alpha}(z)
            \lesssim q^{-\hi{v}}\sum_{z\in\X\setminus U_v} q^{-\alpha\hi{z}} \sum_{m=-\hi{v\wedge z}-1}^{+\infty} q^{-m(\alpha+1)}\\
            &\simeq q^{-\hi{v}} \sum_{z\in\X\setminus U_v} q^{-\alpha\hi{z}} q^{(\alpha+1)\hi{v\wedge z}}\\
            &=q^{-\hi{v}} \sum_{i=-\hi{v}+1}^{+\infty} \sum_{\hi{v\wedge z}=-i}q^{-\alpha\hi{z}} q^{-(\alpha+1)i}\\
            &=q^{-\hi{v}} \sum_{i=-\hi{v}+1}^{+\infty} q^{-(\alpha+1)i}\sigma_{\alpha}(\{z:\hi{v\wedge z}=-i\})\\
            &\simeq q^{-\hi{v}} \sum_{i=-\hi{v}+1}^{+\infty} q^{-(\alpha+1)i}q^{\alpha i}
            \simeq q^{-\hi{v}} q^{\hi{v}}=1,
        \end{align*}
        where the last line follows by the computation in the proof of Lemma \ref{lemma.kib}.
        It follows that
        \begin{equation*}
            \sup_{v\in\X}\,\sup_{x\in U_v} \sum_{z\in\X\setminus U_v} |\mathcal{K}_{\alpha}(z,x)-\mathcal{K}_{\alpha}(z,v)| \sigma_{\alpha}(z)<+\infty,
        \end{equation*}
        that is \eqref{Hc} holds in the case $y=v$. The general case easily follows by the triangular inequality.
    \end{proof}
\end{prop}

The next result is a straightforward consequence of Theorem \ref{thmbound} and Proposition \ref{prophorm}, since $\mathcal K_\alpha$ is clearly self-adjoint.
\begin{cor}
    The horocyclic Bergman projection $P_\alpha$ is of weak type $(1,1)$ and extends to a bounded operator on $L^p_\alpha$ for all $1<p<\infty$. Furthermore, it extends to a bounded operator from $H^1_\alpha$ to $L^1_\alpha$ and from $L^\infty_\alpha$ to $BMO_\alpha$.
\end{cor}

\subsection*{Acknowledgments}
	
 This work is partially supported by the project \lq\lq Harmonic analysis on continuous and discrete structures\rq\rq, which is funded by Compagnia di San Paolo (Cup E13C21000270007). \\Furthermore, the authors are members of the Gruppo Nazionale per l’Analisi Matematica, la Probabilità e le loro Applicazioni (GNAMPA) of the Istituto Nazionale di Alta Matematica (INdAM).

\end{document}